\documentclass[11pt,a4paper,oneside]{article}

\usepackage{graphicx}
\usepackage[T1]{fontenc}

\usepackage{url}
\usepackage{textcomp}
\usepackage[noBBpl]{mathpazo}
\usepackage{booktabs}
\usepackage{amsmath}
\usepackage{amsfonts}
\usepackage{amsthm}
\usepackage{amssymb}
\usepackage[english]{babel}
\usepackage{xcolor}
\usepackage{appendix}
\usepackage{subfigure}

\usepackage{natbib}
\usepackage{hypernat}
\usepackage[colorlinks=true,citecolor=blue!50!black,linkcolor=black,plainpages=false,pdfpagelabels=true]{hyperref}
\usepackage[all]{hypcap}

\graphicspath{{.}{./figures/}{./multiwave/figures/}}

\renewcommand{\leq}{\leqslant} 
\renewcommand{\geq}{\geqslant}

\renewcommand{\epsilon}{\varepsilon} 
\renewcommand{\hat}{\widehat}
\renewcommand{\bar}{\overline}
\renewcommand{\tilde}{\widetilde}

\def\1{\mbox{1\hspace{-.35em}1}}
\def\R{\mathbb{R}}

\def\N{\mathbb{N}}
\def\P{\mathbb{P}}
\def\E{\mathbb{E}}
\def\L{\mathbb{L}}

\def\C{\mathbb{C}}
\def\Z{\mathbb{Z}}

\def\H{\mathbb{H}}

\newcommand{\bff}{\mbox{${\boldsymbol{f}}$}}

\newcommand{\bd}{\mbox{${\mathbf d}$}}

\newcommand{\bC}{\mbox{${\mathbf C}$}}

\newcommand{\bE}{\mbox{${\mathbf E}$}}
\newcommand{\be}{\mbox{${\mathbf e}$}}
\newcommand{\bg}{\mbox{${\mathbf g}$}}

\newcommand{\bi}{\mbox{${\mathbf i}$}}
\newcommand{\bI}{\mbox{${\mathbf I}$}}
\newcommand{\bu}{\mbox{${\mathbf u}$}}

\newcommand{\bA}{\mbox{${\mathbf A}$}}
\newcommand{\bB}{\mbox{${\mathbf B}$}}
\newcommand{\bD}{\mbox{${\mathbf D}$}}
\newcommand{\bQ}{\mbox{${\mathbf Q}$}}
\newcommand{\bL}{\mbox{${\mathbf L}$}}
\newcommand{\bM}{\mbox{${\mathbf M}$}}

\newcommand{\bS}{\mbox{${\mathbf S}$}}

\newcommand{\bG}{\mbox{${\mathbf G}$}}

\newcommand{\bX}{\mbox{${\mathbf X}$}}
\newcommand{\bZ}{\mbox{${\mathbf Z}$}}
\newcommand{\bW}{\mbox{${\mathbf W}$}}

\newcommand{\bm}{\mbox{\boldmath$\mu$}}

\newcommand{\bOmega}{\mbox{\boldmath$\Omega$}}
\newcommand{\bLambda}{\mbox{\boldmath$\Lambda$}}

\newcommand{\bdelta}{\mbox{\boldmath$\delta$}}

\newcommand{\bSigma}{\mbox{\boldmath$\Sigma$}}

\newcommand{\bepsilon}{\mbox{\boldmath$\epsilon$}}

\newcommand{\barG}{\mbox{\boldmath$\overline{G}$}}
\newcommand{\barR}{\mbox{$\overline{R}$}}
\newcommand{\mj}{\mbox{$<\mathcal J>$}}

\newcommand{\diag}{\mbox{diag}}

\newcommand{\argmin}{\displaystyle \mathop{argmin}}

\newcommand{\BigO}{O}

\theoremstyle{plain}
\newtheorem{thm}{Theorem}

\newtheorem{prop}[thm]{Proposition}
\newtheorem{cor}[thm]{Corollary}
\theoremstyle{definition}

\theoremstyle{remark}
\newtheorem*{rem}{Remark}

\providecommand{\keywords}[1]{\noindent \textbf{{Keywords:}} #1}

    \makeatletter
    \let\@fnsymbol\@arabic
    \makeatother

\let\oldFootnote\thanks
\newcommand\nextToken\relax

\renewcommand\thanks[1]{%
    \oldFootnote{#1}\futurelet\nextToken\isFootnote}

\newcommand\isFootnote{%
    \ifx\thanks\nextToken\textsuperscript{~\,,\,}\fi}
    
\begin{document}

\title{Multivariate wavelet Whittle estimation in long-range dependence.}
\author{S. Achard\thanks{Univ. Grenoble Alpes, GIPSA-Lab, F-38000 Grenoble, France}\thanks{CNRS, GIPSA-Lab, F-38000 Grenoble, France}\,~~and I. Gannaz\thanks{Universit\'e de Lyon,
CNRS UMR 5208, INSA de Lyon, Institut Camille Jordan, France}}

\date{October 2015}

\maketitle

\setlength{\parindent}{0pt}

\begin{abstract}{Multivariate processes with long-range dependent properties are found in a large number of applications including finance, geophysics and neuroscience. For real data applications, the correlation between time series is crucial. Usual estimations of correlation can be highly biased due to phase-shifts caused by the differences in the properties of autocorrelation in the processes. To address this issue, we introduce a semiparametric estimation of multivariate long-range dependent processes. The parameters of interest in the model are the vector of the long-range dependence parameters and the long-run covariance matrix, also called functional connectivity in neuroscience. This matrix characterizes coupling between time series. The proposed multivariate wavelet-based Whittle estimation is shown to be consistent for the estimation of both the long-range dependence and the covariance matrix and to encompass both stationary and nonstationary processes. A simulation study and a real data example are presented to illustrate the finite sample behaviour. }
\end{abstract}


 \vspace{10pt}


\setlength{\parskip}{\baselineskip}

\section{Introduction}

The long-range dependence has attracted lots of interest in statistics and in many applications since the seminal paper of Mandelbrot in 1950. First the fractional Brownian motion model was introduced as the unique Gaussian process having stationary increments and self-similarity index $H$ in $(0,1)$ \citep{mandelbrot}. This model is characterized by one parameter called the Hurst exponent. Since then, several extensions were introduced in order to get more complex modellings that better match real data, such as ARIMA, FD, FIN\ldots~We refer to \cite{PercivalWalden} and references therein for an overview of long-range dependence models. These models were used in a large scope of applications, for example finance \citep{Gencay} (see also the references in \cite{Nielsen05survey}), internet traffic analysis \citep{AbryVeitch98}, physical sciences \citep{PercivalWalden, Papanicolaou}, geosciences \citep{WhitcherJensen00} and neuroimagery \citep{maxim.2005.1}.

Nowadays, it is common to record data having multiple sensors, such as neuroimagery (functional Magnetic Resonance Imaging or Electroencephalography). Each sensor records the activity of a specific part of the brain. However, the brain is a complex system with complex interactions between its different parts, so researchers were interested in modelling the sensors as multivariate time series. A similar representation is suited for data acquired in geosciences where, for example, time series correspond to temperatures in several part of the earth, like in \cite{WhitcherJensen00}. For these two applications, it has been shown that the univariate time series present long-range dependence behaviour. Several models accounting for long memory features have been proposed. In \cite{pipiras}, the multivariate Brownian motion was defined. The values of interactions as defined by the covariance matrix must be carefully chosen so that the model is identifiable \citep{mFBM}. Also, the multivariate extension of fractionally difference models was proposed in \cite{Chambers}, which includes the multivariate extension of ARFIMA models with an explicit expression of the short memory terms. In a recent paper, \cite{KechagiasPipiras} highlighted the difficulties to extend the notion of long-range dependence to multivariate time series and proposed specific linear representations of long-range dependence. Concerning multivariate ARFIMA models, \cite{Lobato97, SelaHurvich2008} studied two different classes of extension depending on the order of fractional integration and ARMA models.

Using these long memory models, a typical statistical issue is to estimate the long memory parameter. This characterizes the long-term dependence of the series, which controls many relevant statistical properties. A very large literature exists in the context of univariate time series. First, parametric approaches were considered \citep{FoxTaqqu86,Dahlhaus,Giraitis97rate} which provide fast rates of convergence. However these approaches suffer from inconsistency when the short-term component of the model is misspecified. Semiparametric models were then developed to be robust to model misspecification \citep{Robinson94a, Robinson94b, Robinson95a, Robinson95b}, where the spectral density is modelled only near zero frequency. In the frequency domain, two popular estimators among the semiparametric ones are the Geweke-Porter-Hudak introduced by \cite{Geweke} and the local Whittle estimator of \cite{Robinson95b}. Wavelet-based estimators were also studied, and proved to be adequate for studying fractal time series. In \cite{AbryVeitch98}, the authors developed an estimator using log-regression of the wavelet coefficient variance on the scale index. \cite{Moulines08Whittle} derived the asymptotic properties of a wavelet Whittle estimator.

Considering multivariate fractionally integrated processes \citep{Chambers}, the estimation of memory parameters and covariance matrix have been first studied in \cite{Robinson95a}. Then \cite{Lobato99} proposed a semiparametric two-step estimator. \cite{Shimotsu07} extended this latter approach including phase-shift consideration. \cite{Nielsen11} proposed an extension based on \cite{Abadir}'s extended Fourier transform to estimate long memory parameters for nonstationary time series. In a different approach, \cite{SelaHurvich2012} defined an estimator based on the average periodogram for a power law in coherency.  All these approaches were developed using Fourier log-periodogram. In comparison, there are few wavelet-based estimators of long-range memory parameters in multivariate settings. \cite{FriasRuismedina,WangWang} proposed estimation schemes based on multidimensional wavelet basis. In many real data applications such as geosciences, internet traffic or neurosciences, the number of time series is huge, as the real data example of Section~\ref{sec:real} illustrates. The latter works thus do not seem well adapted. \cite{Achard08} studied a two-dimensional estimation, based on univariate wavelet basis, which defines estimators using a regression of the cross-covariance between the wavelet coefficients. This approach also appears difficult to generalize to high multidimensional settings. The present work defines a wavelet Whittle estimator for multivariate models. The extension to multivariate processes presents two issues. First, a vector of long memory parameters has to be estimated along with the covariance matrix that is modelling the interactions between the time series. Second, as noted in \cite{Robinson94a,Shimotsu07}, the multivariate extension of the fractional integrated model introduces a phase-shift that has to be taken into account in the estimation procedures. The new wavelet-based proposed methodology is shown to be adequate for nonstationary long-range dependence models.

The paper is organised as follows. Section~\ref{sec:model} introduces the specific framework of multivariate long memory processes based on the definition of the spectral density matrix. The multivariate wavelet Whittle estimators of both the long memory parameters and the covariance matrix are defined in Section~\ref{sec:ww}. The properties of this new estimation scheme are derived in Section~\ref{sec:prop} where consistency of both estimations are established. Finally, Section~\ref{sec:simu} presents a simulation study which illustrates that the wavelet Whittle estimators have comparable performances to the Fourier-based ones. In addition, our method provides a very flexible approach to handle both stationary and nonstationary processes. Section~\ref{sec:real} deals with a real data application in neuroscience.

\section{The semiparametric multivariate long-memory framework} 
\label{sec:model}
Let $\bX=\{X_{\ell}(k),k\in\Z, \ell=1,\dots,p\}$ be a multivariate stochastic process. Each process $X_\ell$ is not necessarily stationary. Denote by $\Delta X_\ell$ the first order difference, $(\Delta X_\ell)(k)=X_\ell(k)-X_\ell(k-1)$, and by  $\Delta^{D}X_\ell$ the $D$-th order difference. For every component $X_\ell$, there exists $D_\ell\in\N$ such that the $D_\ell$-th order difference $\Delta^{D_\ell}X_\ell$ is covariance stationary. Following \cite{Chambers,Moulines07SpectralDensity,Achard08}, we consider a long memory process $\bX$ with memory parameters $\bd=(d_1,d_2,\ldots,d_p)$. For any $\bD>\bd-1/2$, we suppose that the multivariate process $\bZ=\diag(\Delta^{D_{\ell}},\ell=1,\dots,p)\bX$ is covariance stationary with a spectral density matrix given by 
\[
\text{for all } (\ell,m)~,~f_{\ell,m}^{(D_\ell,D_m)}(\lambda)=\frac{1}{2\pi}\Omega_{\ell,m}(1-e^{-i\lambda})^{-d_\ell^s}(1-e^{i\lambda})^{-d_m^s}f_{\ell,m}^S(\lambda),\qquad \lambda\in[-\pi,\pi],
\]
where the long memory parameters are given by $d_m^S=d_m-D_m$ for all $m$. The functions $f_{\ell,m}^S(\cdot)$ correspond to the short memory behaviour of the process. 
The generalized cross-spectral density of processes $X_\ell$ and $X_m$ can be written as
\[
f_{\ell,m}(\lambda)=\frac{1}{2\pi}\Omega_{\ell,m}(1-e^{-i\lambda})^{-d_\ell}(1-e^{i\lambda})^{-d_m}f_{\ell,m}^S(\lambda),\qquad \lambda\in[-\pi,\pi].
\]
As it will be explained in Section~\ref{sec:examples}, this model corresponds to a subclass of multivariate long-range dependent time series. In a general case, an additional multiplicative term of the form $e^{i\varphi}$ is required, see {\it e.g.}~\cite{KechagiasPipiras,SelaHurvich2012}.

In the case of a multivariate setting, the spectral density of the multivariate process $\bX$ is thus, 
\begin{equation}\label{eqn:density}\bff(\lambda)=\bOmega\circ( \bLambda^0(\bd) \bff^S(\lambda) \bLambda^0(\bd)^\ast ) ,\quad \lambda\in[-\pi,\pi], \quad \text{~~with~~} \bLambda^0(\bd)=\diag((1-e^{-i\lambda})^{-\bd})\end{equation} 
where $\bd=\bD+\bd^s$. The exponent $\ast$ is the conjugate operator and $\circ$ denotes the Hadamard product. The matrix $\bOmega$ is supposed real symmetric positive definite.

In this semiparametric framework, the spectral density $\bff^S(\cdot)$ corresponds to the short-memory behaviour and the matrix $\bOmega$ is called \textit{fractal connectivity} by \cite{Achard08} or \textit{long-run covariance} matrix by \cite{RobinsonDiscussion}. Similarly to \cite{Moulines07SpectralDensity}, we assume that $\bff^S(\cdot)\in\mathcal H(\beta,L)$ with $0<\beta\leq 2$ and $0<L$. The space $\mathcal H(\beta,L)$ is defined as the class of non-negative symmetric functions $\bg(\cdot)$ on $[\pi,\pi]$ such that $g_{\ell,m}(0)=1$ for all $(\ell,m)\in\{1,\dots,p\}^2$ and that for all $\lambda\in(-\pi,\pi)$, $\max\{|\bg(\lambda)-1|,\, (\ell,m)\in\{1,\dots,p\}^2\}_\infty\,\leq \, L |\lambda|^\beta$. The assumption $f_{\ell,m}^S(0)=1$ for all $(\ell,m)$ is necessary for $\bOmega$ to be identifiable in model~\eqref{eqn:density}.

The spectral density specifies that the two processes $X_\ell$ and $X_m$ have long-memory parameters respectively $d_\ell$ and $d_m$. Parameters with absolute value greater that $1/2$ are allowed, covering nonstationary time series (in this case $D_\ell, D_m \geq 1$). If orders are different, the estimation of the memory parameters is still available but some bias issues occur for the estimation of the underlying covariance $\bOmega$, which is detailled in Section~\ref{sec:ww}.

In order to derive semiparametric estimations of the memory parameters and the matrix $\bOmega$, the term inside the matrix $\bLambda^0(\bd)$ can be simplified using the equality $1-e^{-i\lambda}=2\sin(\lambda/2)e^{i(\pi-\lambda)/2}$. Consequently, when $\lambda$ tends to 0, the spectral density matrix is approximated at first order by
\begin{equation}\label{eqn:approx}\bff(\lambda)\sim  \tilde\bLambda(\bd) \bOmega \tilde\bLambda(\bd)^\ast,\quad \text{~~when~~}\lambda\to 0, \quad \text{~~with~~} \tilde \bLambda(\bd)=\diag(|\lambda|^{-\bd}e^{-i \,\text{sign}(\lambda)\pi \bd/2}).
\end{equation}
Here and subsequently $\sim$ means that the ratio left- and right-hand side converges to one.
 
A similar approximation has been carried out in \cite{Lobato97} or \cite{PhillipsShimotsu04}, while \cite{Shimotsu07} derived a second order approximation. \cite{Lobato99} used $\tilde \bLambda(\bd)=\diag(|\lambda|^{-\bd})$ as an approximation of $\bff(\cdot)$. Whereas \cite{Shimotsu07} chose to approximate $\bff(\cdot)$ using $\tilde \bLambda(\bd)=\diag(\lambda^{-\bd}e^{-i(\pi-\lambda) \bd/2})$, which corresponds to a second order approximation due to the remaining term $\lambda$ in the exponential. As mentioned by \cite{Shimotsu07}, intriguingly, the two defined estimators of long memory parameters are consistent, but only for the estimation of $\bd$. The estimation of the covariance matrix is affected by the choice of  $\tilde \bLambda(\bd)$. In Section~\ref{sec:ww}, we introduce our estimators using approximation~\eqref{eqn:approx}, corresponding to a trade-off between \cite{Lobato99} and \cite{Shimotsu07}. The resulting estimator for $\bd$ is equivalent to the one defined in \cite{Lobato99}. However a specific correction for the estimation of the covariance matrix overcomes the bias caused by the presence of a phase-shift through the complex exponential term. This point has also been raised in the context of detecting cointegration, when the cross-spectral density presents an additional phase parameter comparing to the case studied in this paper.

\subsection{Examples of processes}
\label{sec:examples}

This section provides some examples of processes which satisfy our semiparametric modelling.    

The matrix $\bOmega$ has been defined {\it via} the spectral representation of the process, the link between $\bOmega$ and the covariance of the multivariate process in the temporal space is detailed hereafter.
Let $\bar{\bX}=\frac 1 N \sum_{t=1}^N \bX(t)$ be the empirical mean of the process. If the cross-spectral density is defined and continuous at the frequency $\lambda=0$, Fejer's theorem states that $n^{1/2}\bar{\bX}$ converges in distribution to a zero-mean Gaussian distribution with a covariance matrix equals to $2\pi \bff(0)$. When the cross-spectral density satisfies an approximation~\eqref{eqn:density}, \cite{RobinsonDiscussion} indicates that 
\[
\bD_n \E(\bar{\bX}\bar{\bX}^T)\bD_n \xrightarrow[n\to \infty]{} 2\pi \bOmega\circ \bQ(\bd)
\]
where $\bD_n=\diag(n^{1/2-\bd})$ and $Q_{\ell,m}(\bd)=\frac{\sin(\pi d_\ell)+\sin(\pi d_m)}{\Gamma(d_\ell+d_m+2)\sin(\pi(d_\ell+d_m))}$.
The exponent $T$ denotes the transpose operator.

\subsubsection{Causal linear representations of Kechagias and Pipiras (2014)} 
\label{sec:causal_rep}

\cite{KechagiasPipiras} define Long-Range Dependence (LRD) with a more general setting for the phase. Let $\bX$ be a $p$-multivariate time series. We suppose that $\bX$ is second-order stationary and that it admits a spectral density.
The time series $\bX$ is long-range dependent in the sense of \cite{KechagiasPipiras} if its spectral density $\bff(\cdot)$ satisfies:
\[
\bff(\lambda)=\frac{1}{2\pi}\diag(|\lambda|^{-\bd})\,\bG(\lambda)\,\diag(|\lambda|^{-\bd})
\] 
with $\bd=\begin{pmatrix}d_1 & \dots & d_p \end{pmatrix}\in(0,1/2)^p$
and $\bG(\cdot)$ a $\C^{p\times p}$-valued Hermitian non negative definite matrix function satisfying 
\begin{equation*}
\bG(\lambda)\sim_{\lambda\to 0^+} \bG=(\Omega_{\ell,m}e^{i\varphi_{\ell,m}})_{\ell,m=1,\dots,p}
\end{equation*} 
with $\Omega_{\ell,m}\in\R$, $\ell,m=1,\dots,p$ and $\varphi_{\ell,m}\in(-\pi,\pi]$. The phases $\varphi_{\ell,m}$ measure the dissymmetry of the process $\bX$ at large lags. Indeed the specificity of multivariate time series is that the autocovariance function may no longer be symmetric compared with the univariate framework ({\it i.e.} $\gamma(-h)=\gamma(h)^T$ may not be equal to $\gamma(h)$). When the process is time reversible, the  autocovariance function is symmetric. Time reversible series will satisfy $\varphi_{\ell,m}=0$ for all $\ell,m=1,\dots,p$. We refer to proposition 2.1 of \cite{KechagiasPipiras} that gives some highlights on the phase parameters. In Proposition 3.1 \cite{KechagiasPipiras} give examples of LRD linear time series where any combinations of $(\bd,\bG)$ can be chosen.

Many results of the present work can be generalized to LRD processes of \cite{KechagiasPipiras}. Yet as our goal is to recover the matrix $\bOmega$, an assumption on the form of the phase $\varphi$ is necessary (since we consider a real filter which deletes the imaginary part). The most widely used definition of phase is $\varphi_{\ell,m}=-\frac{\pi}{2}(d_\ell-d_m)$ which includes a large scope of models, see section \ref{sec:ARFIMA}.

Such a definition of phase is verified, for example, using a causal representation of processes described in \cite{KechagiasPipiras}. Let $\{\bepsilon_k\}_{k\in\Z}$ be a $\R^p$-valued white noise, satisfying $\E[\bepsilon_k] = 0$ and $\E[\bepsilon_k\bepsilon_k^T]=\bI_p$. Let also $\{\bB_k = (B_{\ell,m,k})_{\ell,m=1,\dots,p}\}_{k\in\N}$ be a sequence of real-valued matrices such that $B_{\ell,m,k} = L_{\ell,m}(k)\,k^{d_a-1}, ~ k\in\N,$ where $d_\ell\in(0, 1/2)$ and $\bL(k)$, $k=1,...,p$, is an $\R^{p\times p}$-valued function satisfying $\bL(k) \sim_{k\to +\infty} \bA$
for some $\R^{p\times p}$-valued matrix $\bA$. We define the time series $\bX$ given by the causal linear representation 
\[
\bX(k) = \sum_{j=0}^{+\infty} \bB_{j}\bepsilon_{k-j}.
\]
 Corollary 4.1 of \cite{KechagiasPipiras} states that the process $\bX$ is LRD with 
\begin{align*}
\Omega_{\ell,m}&=\frac{\Gamma(d_\ell)\Gamma(d_m)}{2\pi}(\bA \bA^\ast )_{\ell,m},\\
\varphi_{\ell,m} &= -\frac{\pi}{2}(d_\ell-d_m).
\end{align*}
Such causal linear representations thus satisfies~\eqref{eqn:approx}. An example of such a representation is the multivariate $ARFIMA(0,\bd,0)$ model presented in next subsection.

\subsubsection{Multivariate ARFIMA of Lobato (1997)}
\label{sec:ARFIMA}

The composition of linear filters does not commute in the multivariate case. Consequently there are multiple extensions of univariate ARFIMA to the multivariate framework. We detail in this section the multivariate ARFIMA models of \cite{Lobato97}.

Let $\bu$ be a $p$-dimensional white noise with $\E[\bu(t)\mid \mathcal F_{t-1}]=0$ and $\E[\bu(t)\bu(t)^T\mid \mathcal F_{t-1}]=\bSigma$, where $\mathcal F_{t-1}$ is the $\sigma$-field generated by $\{\bu(s),\,s<t\}$, and $\bSigma$ is a positive definite matrix. The spectral density of $\bu$ satisfies $\bff_u(\lambda)=\bSigma/(2\pi)$.

Let $(\bA_k)_{k\in\N}$ be a sequence of $\R^{p\times p}$-valued matrices with $\bA_0$ the identity matrix and $\sum_{k=0}^\infty \|\bA_k\|^2<\infty$. Let $\bA(\cdot)$ be the discrete Fourier transform of the sequence, $\bA(\lambda)=\sum_{k=0}^{\infty} \bA_k e^{i k\lambda}$. We assume that all the roots of $|\bA(\L)|$ are outside the closed unit circle.

\cite{Lobato97} defines two multivariate ARFIMA models which both satisfy approximation~\eqref{eqn:approx}.
\begin{description}
\item[Model A.] Let $\bX$ be defined by $\bA(\L)\,\diag(\1-\L)^{\bd}\,\bX(t)=\bB(\L)\bu(t)$. The spectral density satisfies 
\[
f_{\ell,m}(\lambda)\sim_{\lambda\to 0^+} \frac{1}{2\pi}\Omega_{\ell,m}e^{-i\pi/2(d_\ell-d_m)}\lambda^{-(d_\ell+d_m)}
\] 
with $\bOmega=\bA(1)^{-1}\bB(1)\bSigma\bB(1)^T{\bA(1)^T}^{-1}.$ It is straightforward that Model A statifies approximation~\eqref{eqn:approx}.
\item[Model B.] Let $\bX$ be defined by $\diag(\1-\L)^{\bd}\,\bA(\L)\,\bX(t)=\bB(\L)\bu(t)$. The spectral density satisfies 
\[
f_{\ell,m}(\lambda)\sim_{\lambda\to 0^+} \frac{1}{2\pi}\sum_{a,b} \beta_{a,b} \alpha_{\ell,a}\alpha_{m,b}e^{-i\pi/2(d_a-d_b)}\lambda^{-(d_a+d_b)}
\]
with $\alpha_{\ell,m}=(\bA(1)^{-1})_{\ell,m}$ and $\beta_{\ell,m}=(\bB(1)\bSigma\bB(1)^T)_{\ell,m}$. In Model B the spectral density is equivalent around the zero frequency to the term in $(a,b)=argmax \{|d_a+d_b|, \; \beta_{a,b} \alpha_{\ell,a}\alpha_{m,b}\neq 0\}$. It gives a more general setting $f_{\ell,m}(\lambda)\sim_{\lambda\to 0^+} G_{\ell,m}\lambda^{-d_{\ell,m}}$ with $|d_{\ell,m}|\leq (d_\ell+d_m)/2$. (This model is studied more extensively in \cite{SelaHurvich2012}, where the authors propose an estimator for $d_{1,2}$ in a bivariate framework.) Setting $\Omega_{\ell,m}=0$ if $|d_{\ell,m}|< (d_\ell+d_m)/2$ equation~\eqref{eqn:approx} holds. This means that if there is cointegration, the corresponding long-run covariance value is set to zero.
\end{description}
In particular, Model A and Model B of \cite{Lobato97} include FIVAR and VARFI models of \cite{SelaHurvich2008}. These multivariate ARFIMA models admit a causal linear representation. They include short-range dependence behaviour, through the terms $A(\L)$ and $B(\L)$. When these terms are equal to identity, we obtain an $ARFIMA(0,\bd,0)$ satisfying the causal linear representation given in subsection \ref{sec:causal_rep}. This (causal) multivariate $ARFIMA(0,\bd,0)$ is a subclass of \cite{KechagiasPipiras}'s (possibly non causal) definition.

\begin{rem} In the univariate setting, when $A(\cdot)$ and $B(\cdot)$ have no common zeros, \cite{KokoszkaTaqqu95} establish that the time series $X$ admits a linear representation $X(k)=\sum_{j=0}^{+\infty} C_j\epsilon_{k-j}$ with $C_j=\frac{B(1)}{A(1)\Gamma(d)}j^{d-1}+\BigO(j^{-1})$ when $j$ goes to infinity. The terms of the linear representation thus satisfy approximately the condition of \cite{KechagiasPipiras}. The extension to multivariate setting would yield information about the link between \cite{Lobato97}'s models and \cite{KechagiasPipiras}'s condition on causal representations. However, it has not been explored in this paper since we will not use this fact in any essential way. 
\end{rem}

\section{Multivariate wavelet Whittle estimation}
\label{sec:ww}

This section first defines the wavelet transform of the processes and then gives some results on the cross behaviour of the wavelet coefficients. The main point is the presence of a phase-shift caused by the differences in the long-memory parameters. Finally the proposed estimation scheme is derived, defining simultaneous estimators of the long-memory parameters and of the long-run covariance, which takes into account the phase-shift (based on the first order approximation~\eqref{eqn:approx}).

\subsection{The wavelet analysis}

Let $(\phi(\cdot),\psi(\cdot))$ be respectively a father and a mother wavelets. Their Fourier transforms are given by $\hat \phi(\lambda)=\int_{-\infty}^\infty \phi(t)e^{-i\lambda t}dt$ and $\hat \psi(\lambda)=\int_{-\infty}^\infty \psi(t)e^{-i\lambda t}dt$. 

At a given resolution $j\geq 0$, for $k\in\Z$, we define the dilated and translated functions $\phi_{j,k}(\cdot)=2^{-j/2}\phi(2^{-j}\cdot -k)$ and $\psi_{j,k}(\cdot)=2^{-j/2}\psi(2^{-j}\cdot -k)$. Thoughout the paper, we adopt the same convention as in \cite{Moulines07SpectralDensity} and \cite{Moulines08Whittle}, that is large values of the scale index $j$ correspond to coarse scales (low frequencies).

Let $\tilde{\bX}(t)=\sum_{k\in\Z}\bX(k)\phi(t-k)$. The wavelet coefficients of the process $\bX$ are defined by 
\[
\bW_{j,k}=\int_\R \tilde{\bX}(t)\psi_{j,k}(t)dt\quad j\geq 0, k\in\Z.
\]
For given $j\geq 0$ and $k\in\Z$, $\bW_{j,k}$ is a $p$-dimensional vector $\bW_{jk}=\begin{pmatrix}
W_{j,k}(1) & W_{j,k}(2) & \dots & W_{j,k}(p) \end{pmatrix}$, where $W_{j,k}(\ell)= \int_\R \tilde{X_\ell}(t)\psi_{j,k}(t)dt$.

The regularity conditions on the wavelet transform are expressed in the following assumptions. They will be needed throughout the paper.

\begin{description}
\item[(W1)] The functions $\phi(\cdot)$ and $\psi(\cdot)$ are integrable, have compact supports, $\int_\R \phi(t)dt=1$ and $\int \psi^2(t)dt=1$;
\item[(W2)] There exists $\alpha>1$ such that $\sup_{\lambda\in\R} |\hat \psi(\lambda)|(1+|\lambda|)^\alpha\,<\,\infty$, {\it i.e.} the wavelet is $\alpha$-regular;
\item[(W3)] The mother wavelet $\psi(\cdot)$ has $M>1$ vanishing moments.
\item[(W4)] The function $\sum_{k\in\Z}k^\ell\phi(\cdot -k)$ is polynomial with degree $\ell$ for all $\ell=1,\ldots,M-1$.
\item[(W5)] For all $i=1,\dots,p$, $(1+\beta)/2-\alpha\,<\,d_i\,\leq\, M$.
\end{description}

These conditions are not restrictive, and many standard wavelet basis satisfy them. Among them, Daubechies wavelets are compactly supported wavelets parametrized by the number of vanishing moments $M$. They are $\alpha$-regular with $\alpha$ an increasing function of $M$ going to infinity (see \cite{Daubechies}). Assumptions (W1)-(W5) will hold for Daubechies wavelet basis with sufficiently large $M$. 

\begin{rem} 
The couple of functions $(\phi(\cdot),\psi(\cdot))$ can be associated with a multiresolution analysis, but this condition is not necessary. Similarly, the orthogonality of the family $\{\psi_{j,k}(\cdot)\}$ is not required. See \cite{Moulines07SpectralDensity}, Section 3.
\end{rem}

Under assumption (W3), the wavelet transform performs an implicit differentiation of order $M$. Thus it is possible to apply it on nonstationary processes. In Fourier analysis, tapering procedures are necessary to consider directly nonstationary frameworks, see {\it e.g.}~\cite{VelascoRobinson00} and references therein. Some recent works propose a procedure that differentiates the data before tapering (\cite{HurvichChen00} and references therein). Another extension of Fourier to nonstationary frameworks has been proposed by \cite{Abadir} and applied by \cite{Nielsen11} in multivariate analysis.

In practice, a finite number of realisation of the process $\bX$, say $\bX(1),\ldots \bX(N)$, is observed.  
Since the wavelets have a compact support only a finite number $n_j$ of coefficients are non null at each scale $j$. Suppose without loss of generality that the support of $\psi(\cdot)$ is included in $[0,T_{\psi}]$ with $T_{\psi}\geq 1$. For every $j\geq 0$, define \begin{equation}\label{eqn:nj} 
n_j:=  \max{(0,2^{-j}(N-T_{\psi}+1))}.
\end{equation} 
Then for every $k< 0$ and $k> n_j$, the coefficients $\bW_{j,k}$ are set to zero because all the observations are not available. In the following, $n=\sum_{j=j_0}^{j_1} n_j$ denotes the total number of non-zero coefficients used for estimation. 

\subsection{Spectral approximation of wavelet coefficients}

Let us first recall some results of \cite{Moulines07SpectralDensity} for the wavelet transform of a univariate process. Let $W_{j,k}$ denote the wavelet coefficient of a unidimensional process $X$, with spectral density $ f(\lambda)=|1-e^{i\lambda}|^{-2d_0}f^S(\lambda)$, where $d_0 \in \R$ (note that $d_0$ can be outside of the interval $[-1/2,1/2]$). \cite{Moulines07SpectralDensity} state that under assumptions (W1)-(W5), the wavelet coefficients process $(W_{j,k})_{k\in\Z}$ is covariance stationary for any given $j\geq 0$. However they also stress that the between-scale coefficients are not decorrelated. It is shown that the wavelet coefficients are decorrelated when the wavelet bases is orthonormal and $d_0=0$ but it is not valid in general settings. Many proposition of estimators of long-memory can be found in \cite{WaveletWhittleWornell92, AbryVeitch98, Jensen99, Gonzaga}, among others. These works assume that the wavelet coefficients are decorrelated. We follow \cite{Bardet00} or \cite{Moulines08Whittle} in taking into account the within and between scales behaviour.

Let $j\geq 0$ and $j'=j-u\leq j$ be two given scales. Following \cite{Moulines07SpectralDensity}, the between-scale process is defined as the sequence $\{W_{j,k},W_{j-u,2^{u}k+\tau}, \tau=0,\dots 2^u-1\}_{k\in\Z}$. Let $\bD_{j,u}(\cdot;d_0)$ be the cross-spectral density between $\{W_{j,k}\}_{k\in\Z}$ and $\{W_{j-u,2^{u}k+\tau}, \tau=0,\dots 2^u-1\}_{k\in\Z}$. For any $\lambda\in(-\pi,\pi)$, $\bD_{j,u}(\lambda;d_0)$ is a $2^u$-dimensional vector. Theorem~1 in \cite{Moulines07SpectralDensity} establishes that under assumptions (W1)-(W5) there exists a positive constant $C$ such that for all $\lambda\in(-\pi,\pi)$,  \[
\left|\bD_{j,u}(\lambda;d_0)-2^{2jd_0}\bD_{\infty,u}(\lambda;d_0)\right|\leq C 2^{j(2d_0-\beta)},
\]
where 
\[
\bD_{\infty,u}(\lambda;d_0):=\sum_{t\in\Z}|\lambda+2t\pi|^{-2d_0}\hat\psi^\ast(\lambda+2t\pi)2^{-u/2}\hat\psi(2^{-u}(\lambda+2t\pi))\be_u(\lambda+2t\pi)
\]
with $\be_u(\xi)=\begin{pmatrix}
1& e^{-i2^{-u}\xi}&\dots &e^{-i2^{-u}(2^u-1)\xi}
\end{pmatrix}^T.$

The key point of our estimation is the extension of results obtained by \cite{Moulines07SpectralDensity} to the multivariate framework. Due to the complexity of the multivariate setting, we choose not to characterize the behaviour of the wavelet coefficients in terms of cross-spectral densities.

First, in order to extend the results of \cite{Moulines07SpectralDensity} to a multivariate framework, the covariance behaviour of ${\bW_{j,k}}$ for given $(j,k)$ is derived.
Let $\theta_{\ell,m}(j)$ denote the wavelet covariance at scale $j$ between processes $X_\ell$ and $X_m$, $\theta_{\ell,m}(j)=Cov(W_{j,k}(\ell), W_{j,k}(m))$ for any position $k$. Using the spectral density representation, $\theta_{\ell,m}(j)$ satisfies 
\[
\theta_{\ell,m}(j)=\int_{-\pi}^{\pi}{ (1-e^{-i\lambda})^{-d_\ell}(1-e^{i\lambda})^{-d_m}\Omega_{\ell,m}f_{\ell,m}^S(\lambda)|\mathbb{H}_j(\lambda)|^2\,d\lambda},
\]
where $\mathbb{H}_j$ is the gain function of the wavelet filter.\\

The following proposition establishes a second order approximation of the spectral density in a neighbourhood of zero, such as the one derived in \cite{Shimotsu07}.  

\begin{prop}\label{prop:cov_ondelettes0}
Let assumptions (W1)-(W5) hold.
Let $j\geq0$, $\ell, m=1,\dots,p$. Let $K_j$ be defined by 
\[
K_j(d_\ell,d_m)=\int_{-\infty}^{\infty}{|\lambda|^{-(d_\ell+d_m)}\cos(2^{-j}{\lambda(d_\ell-d_m)/2})|\hat\psi(\lambda)|^2\,d\lambda}.
\]
Then there exists a constant $C_0$ depending on $\beta, \min_i d_i, \max_i d_i, \bOmega$, $\phi(\cdot)$ and $\psi(\cdot)$ such that
 \begin{equation}
 \label{eqn:theta_cos}
|\theta_{\ell,m}(j) - \Omega_{\ell,m} 2^{j(d_\ell+d_m)}\cos(\pi(d_\ell-d_m)/2)K_j(d_\ell+d_m)|\leq C_0 L 2^{(d_\ell+d_m-\beta)j}. 
\end{equation}
 \end{prop}
 
Note that the second order approximation of the spectral density depends on $j$ also through the function $K_j$. The following proposition is deriving a first order approximation so that its logarithm is linear in $j$.   

\begin{prop}\label{prop:cov_ondelettes}
Let $K(\cdot)$ be defined by 
\[
K(\delta)=\int_{-\infty}^{\infty}{|\lambda|^{-\delta}|\hat\psi(\lambda)|^2\,d\lambda},\quad\delta\in(-\alpha,M).
\]
Under assumptions (W1)-(W5), there exists a constant $C$ depending on $\beta,\min_i d_i, \max_i d_i, \bOmega$, $\phi(\cdot)$ and $\psi(\cdot)$ such that, for all $j\geq 0$, for all $\ell, m=1,\dots,p$,
  \begin{equation}
 \label{eqn:theta}
 |\theta_{\ell,m}(j) - \Omega_{\ell,m} 2^{j(d_\ell+d_m)}\cos(\pi(d_\ell-d_m)/2)K(d_\ell+d_m)|\leq C L 2^{(d_\ell+d_m-\beta)j}.
 \end{equation}
\end{prop}

Observe that the approximation \eqref{eqn:theta} shows that the difficulty with the case of multivariate long-memory processes is the appearance of a phase-shift that has to be taken into account for the estimation of the covariance $\bOmega$. Indeed, $\theta_{\ell,m}(j)$ is proved to be close to a term proportional to $\cos(\pi(d_\ell-d_m)/2)$. Then, if $d_\ell\in[-1/2,1/2]$ and $d_m=2k+1+d_\ell^S$ with $k\in\N$, Proposition~\ref{prop:cov_ondelettes} implies that for all $j$, $\theta_{\ell,m}(j)$ is negligible, meaning that the covariance of the wavelet coefficients is close to zero. Consequently, using the covariance of the wavelet coefficients does not allow to estimate the matrix $\bOmega$ accurately. This example corresponds to a covariance-stationary process $X_\ell$ and a process $X_m$ such that $\Delta X_m$ is covariance stationary, both with the same long-memory parameter $d_\ell$. We show in what follows that the consistency of the long memory parameters is not affected by bias in the estimation of $\bOmega$.

The covariance behaviour for the between scale process is derived in the following proposition.
\begin{prop}\label{prop:cov_croisee}

For all $j\geq 0$, $u\geq 0$ and $\lambda\in(-\pi,\pi)$, we define  \begin{align*}
D^{(j)}_{u;\tau}(\lambda;f_{\ell,m}(\cdot))&=\sum_{t\in\Z}f_{\ell,m}(2^{-j}(\lambda+2t\pi))2^{-j}\H_j(2^{-j}(\lambda+2t\pi))\H_{j-u}^\ast(2^{-j}(\lambda+2t\pi))e^{-i2^{-u}\tau(\lambda+2t\pi)}\\
\tilde D_{u,\tau}(\lambda;\delta)&=\sum_{t\in\Z} |\lambda+2t\pi|^{-\delta}\hat\psi^\ast(\lambda+2t\pi)2^{-u/2}\hat\psi(2^{-u}(\lambda+2t\pi))e^{-i2^{-u}\tau(\lambda+2t\pi)}
\end{align*}
and $
K_{u,\tau}(v;\delta)=\int_{-\pi}^{\pi}{\tilde D_{u,\tau}(\lambda;\delta)e^{i\lambda v}\,d\lambda}.
$

Then for all $j\geq 0$, for all $u,v\geq 0$, $\tau=0,\dots,2^u-1$,
\[
Cov(W_{j,k}(\ell),W_{j-u,2^{-u}k'+\tau}(m))=\int_{-\pi}^\pi D^{(j)}_{u;\tau}(\lambda;(\ell,m))e^{i\lambda(k-k')}d\lambda.
\]
Under assumptions (W1)-(W5), there exists a constant C depending on $\beta,\min_i d_i, \max_i d_i, \bOmega, \phi(\cdot)$ and $\psi(\cdot)$ such that, for all $j\geq 0$, for all $u,v\geq 0$, $\tau=0,\dots,2^u-1$, for all $\lambda\in(-\pi,\pi)$,
\begin{equation*}
\left|D^{(j)}_{u;\tau}(\lambda;f_{\ell,m}(\cdot)) - \Omega_{\ell,m} 2^{j(d_\ell+d_m)}\cos({\pi(d_\ell-d_m)/2})\tilde D_{u,\tau}(\lambda;d_\ell+d_m)\right|\\ \leq C L 2^{(d_\ell+d_m-\beta)j}.
\end{equation*}
and
\begin{multline*}
\left|Cov[W_{j,k}(\ell)W_{j-u,2^{-u}k'+\tau}(m)] - \Omega_{\ell,m} 2^{j(d_\ell+d_m)}\cos({\pi(d_\ell-d_m)/2})K_{u,\tau}(k-k';d_\ell+d_m)\right|\\ \leq C L 2^{(d_\ell+d_m-\beta)j}.
\end{multline*}
\end{prop}

When $u=0$ and $2^uk'+\tau=k$, the quantity $K_{0,0}(0;d_\ell+d_m)$ is equal to $\int_{-\infty}^\infty{|\lambda|^{-(d_\ell+d_m)} |{\hat\psi(\lambda)}|^2 }d\lambda$. Let us remark that $K_{0,0}(0;\cdot)$ is equal to the function $K(\cdot)$ defined in Proposition~\ref{prop:cov_ondelettes}.

\subsection{Wavelet Whittle estimation}

Let $j_1 \geq j_0 \geq 1$ be respectively the maximal and the minimal resolution levels that are used in the estimation procedure. The estimation is based on the vectors of wavelets coefficients $\left\{\bW_{j,k},\, j_0\leq j\leq j_1,\, k\in\Z\right\}$.

The wavelet Whittle approximation of the negative log-likelihood is denoted by $\mathcal L(\cdot)$. The criterion corresponds to the negative log-likelihood of Gaussian vectors $(W_{j,k}(\ell))_{j,k,\ell}$. \cite{Hannan, FoxTaqqu86} prove that the Whittle approximation is giving satisfactory results for nongaussian processes.
In our framework, the wavelet Whittle criterion is defined as,
\begin{equation} \label{eqn:whittle}
\mathcal L(\bG(\bd),\bd) = \frac{1}{n}\sum_{j=j_0}^{j_1} \left[ n_j \log\det\left(\bLambda_j(\bd)\bG(\bd)\bLambda_j(\bd)\right)+\sum_{k=0}^{n_j} \bW_{j,k}^T\left(\bLambda_j(\bd)\bG(\bd)\bLambda_j(\bd)\right)^{-1}\bW_{j,k}\right],
\end{equation}

where $\bLambda_j(\bd)$ and the matrix $\bG(\bd)$ are obtained with Proposition~\ref{prop:cov_ondelettes},
\[
\bLambda_j(\bd)=\diag\left(2^{j\bd}\right)
\]
and the $(\ell,m)$-{th} element of the matrix $\bG(\bd)$ is $G_{\ell,m}(\bd)=\Omega_{\ell,m}K(d_\ell+d_m)cos({\pi(d_\ell-d_m)/2})$.

For each $j\geq 0$, the quantity $\sum_k\bW_{j,k}^T\left(\bLambda_j(\bd)\bG(\bd)\bLambda_j(\bd)\right)^{-1}\bW_{j,k}$ has a dimension equal to 1 and is equal to its trace. Thus, \begin{equation}
\label{eqn:critere}
\mathcal L(\bG(\bd),\bd) = \frac{1}{n}\sum_{j=j_0}^{j_1} \left[n_j  \log\det\left(\bLambda_j(\bd)\bG(\bd)\bLambda_j(\bd)\right)+\text{trace}\left(\left(\bLambda_j(\bd)\bG(\bd)
\bLambda_j(\bd)\right)^{-1}\bI(j)\right)\right],
\end{equation}
where $\bI(j)=\sum_{k=0}^{n_j} \bW_{j,k} \bW_{j,k}^T$. 
Note that this expression is very similar to the multivariate Fourier Whittle estimator of \cite{Shimotsu07}. Here we replace the periodogram by the wavelet scalogram $\bI(j)$. 

\begin{rem} In Fourier analysis, {\it e.g.}~\cite{Shimotsu07}, the periodogram is normalized. In wavelet analysis, the normalization factor may depend on the resolution $j$, and the scalogram is not normalized. For every $j$ the scalogram $\bI(j)$ should be normalized by $n_j$. In the remainder of the paper, we will keep the initial $\bI(j)$ for convenience.
\end{rem}

Deriving expression~\eqref{eqn:critere} with respect to the matrix $\bG$ yields
\[
\frac{\partial \mathcal L}{\partial \bG}(\bG,\bd)=\frac{1}{n}\sum_{j=j_0}^{j_1}\left[n_j \bG^{-1}-\bG^{-1}\bLambda_j(\bd)^{-1}
\bI(j)\bLambda_j(\bd)^{-1} \bG^{-1}\right].
\]
Hence, the minimum for fixed $\bd$ is attained at \begin{equation}\label{eqn:G}
\hat \bG(\bd) =\frac{1}{n} \sum_{j=j_0}^{j_1} \bLambda_j(\bd)^{-1}
\bI(j)\bLambda_j(\bd)^{-1}.
\end{equation}
Replacing $\bG(\bd)$ by $\hat \bG(\bd)$, the objective criterion is defined as 
\begin{align}
R(\bd)&:=\mathcal L(\hat \bG(\bd),\bd)-1\nonumber\\
&= \log\det(\hat \bG(\bd)) + \frac{1}{n}\sum_{j=j_0}^{j_1}  n_j \log(\det\left(\bLambda_j(\bd)\bLambda_j(\bd)\right), \nonumber \\
\label{eqn:R}&= \log\det(\hat \bG(\bd)) + 2\log(2)\left(\frac{1}{n}\sum_{j=j_0}^{j_1} j n_j\right)\left(\sum_{\ell=1}^p d_\ell\right).
\end{align}
The vector of the long-memory parameters $\bd$ is estimated by $\hat \bd=\argmin_{\bd} R(\bd)$. The estimator $\hat \bd$ is exactly equal to the one introduced in \cite{Moulines08Whittle} when the matrix $\bOmega$ is diagonal corresponding to univariate setting. 

In a second step of estimation we define $\hat \bG(\hat \bd)$, estimator of $\bG(\bd)$. Finally applying the correction of phase-shift yields the estimation of the covariance matrix $\bOmega$ 
\begin{equation}\label{eqn:Ohat}\hat\Omega_{\ell,m}=\hat G_{\ell,m}(\hat \bd)/(\cos(\pi(\hat d_\ell-\hat d_m)/2) K(\hat d_\ell+\hat d_m)).\end{equation}
Equation~\eqref{eqn:Ohat} is correctly defined as the probability that $\hat d_\ell-\hat d_m$ is exactly congruent to 1 {\it modulo} 2 is null. Consequently estimator $\hat\bOmega$ is defined almost surely. Yet, empirically when $d_\ell-d_m$ is close to 1 {\it modulo} 2, the estimation of $\bOmega$ may be strongly biased.

\section{Main results}
\label{sec:prop}

In the above, we have defined the MWW estimator, the following section deals with the asymptotic behaviour of the estimators. The consistency of the estimators is established, under a condition which controls the variance of the empirical wavelet cross-covariances. The first part of this section introduces this condition and characterizes a class of processes for which it is satisfied. The second part details the asymptotic results of convergence.

\subsection{Additional condition}

The following condition is an additional assumption, which gives an asymptotic control of the wavelet scalogram.

{\bf \large Condition (C)}
\[
\text{For all } \ell,m=1,\ldots,p,\quad \sup_n \sup_{j\geq 0} \frac{1}{n_j 2^{2j(d_\ell+d_m)}} Var\left({I_{\ell,m}(j)}\right) \,<\,\infty
\]

This condition is slightly more restrictive than condition~(9) of \cite{Moulines08Whittle} in a univariate framework, where their spectral density of the process is only defined on a neighbourhood of zero. 

The following proposition gives a class of multivariate processes such that Condition~(C) holds.

\begin{prop}
\label{prop:condition}
Suppose that there exists a sequence $\{\bA(u)\}_{u\in\Z}$ in $\R^{p\times p}$ such that $\sum_u \|\bA(u)\|_\infty^2<\infty$ and 
\[
\forall t,\,~~~\Delta^{D} \bX(t)=\bm+\sum_{u\in\Z} \bA({t+u}) \bepsilon(t)
\]
with $\bepsilon(t)$ weak white noise process, in $\R^p$. Let $\mathcal F_{t-1}$ denote the $\sigma$-field of events generated by $\{\bepsilon(s), \,s\leq t-1\}$. Assume that $\bepsilon$ satisfies $\E[\bepsilon(t)|\mathcal F_{t-1}]=0$, $\E[\epsilon_a(t)\epsilon_b(t)|\mathcal F_{t-1}]=\1_{a=b}$ and $\E[\epsilon_a(t)\epsilon_b(t)\epsilon_c(t)\epsilon_d(t)|\mathcal F_{t-1}]=\mu_{a,b,c,d}$ with $|\mu_{a,b,c,d}|\leq \mu_\infty<\infty$, for all $a,b,c,d=1,\ldots,p$.\\
Then, under assumptions (W1)-(W5), Condition (C) holds.

\end{prop}

The proof is given in appendix~\ref{proof:conditionC}.

This assumption of a Cramer-Wold type decomposition of the process $\bX$ with a linear fourth-order stationary process was made among others by \cite{Lobato99}, \cite{Shimotsu07}, \cite{Giraitis97rate}, or Theorem~1 of \cite{Moulines08Whittle}. 
As discussed in \cite{Lobato99}, there exist models with density~\eqref{eqn:density}, where the condition (C) is not satisfied, however, it is not particularly restrictive.

\subsection{Convergence}

We suppose that we have $N$ observations of a multivariate $p$-vector process $\bX$, namely $\bX(1),\dots \bX(N)$ with a spectral density satisfying approximation~\eqref{eqn:approx} around the zero frequency. For given functions $(\phi(\cdot),\psi(\cdot))$, and for given levels $0\leq j_0\leq j_1$, the estimator of $\bd$ is the argument minimizing $R$, defined by~\eqref{eqn:R}, and the matrix $\bG$ is estimated by $\hat \bG(\hat \bd)$ defined by~\eqref{eqn:G}. From now on, we will add the superscript $0$ to denote the true parameter values, $\bd^0$ and $\bG^0$. 

The following assumptions on the resolution levels $j_0$ and $j_1$ will be needed throughout the paper. We assume that either the difference $j_1-j_0$ is constant or it tends to infinity as $N$ tends to infinity.

The following result shows the consistency of the estimators and the rate of convergence. The proofs are given in Appendix.

\begin{thm}
\label{prop:convergence}
Assume that (W1)-(W5) and Condition (C) hold. If $j_0$ and $j_1$ are chosen such that $2^{-j_0\beta}+N^{-1/2} 2^{j_0/2}\to 0$ and $j_0 < j_1 \leq j_N$ with $j_N=\max\{j,n_j\geq 1\}$, then 
\[
\hat \bd -\bd^0=o_\P(1).
\]
\end{thm}
This result generalises that of \cite{Moulines08Whittle}. It deals with multivariate settings, with the same assumption on the wavelet filter and on the choice of the scale $j_0$. The condition in Theorem~\ref{prop:convergence} is equal to the one obtained in Proposition 9 of \cite{Moulines08Whittle}, in the univariate case, that is $1/j_0+N^{-1/2} 2^{j_0/2}\to 0$. We choose here to express the condition with the parameter $\beta$ since the optimal choice of $j_0$ depends on $\beta$ as it is shown in Corollary~\ref{prop:optimalite} below.

The convergence of $\hat\bG(\hat d)$ to $\bG^0$ is not established under assumptions of Theorem~\ref{prop:convergence}. However, we prove it in the following theorem, under more restrictive conditions.

\begin{thm}
\label{prop:vitesse}
Assume that (W1)-(W5) and Condition (C) hold. If $j_0$ and $j_1$ are chosen such that $\log(N)^2(2^{-j_0\beta}+ N^{-1/2} 2^{j_0/2})\to 0$ and $j_0 < j_1 \leq j_N$ then 
\[
\hat \bd-\bd^0=\BigO_\P(2^{-j_0\beta}+ N^{-1/2} 2^{j_0/2}),
\]
\begin{align*}
\forall(\ell,m)\in\{1,\ldots,p\}^2,\,\hat G_{\ell,m}(\hat \bd)-G_{\ell,m}(\bd^0) &= \BigO_\P(\log(N)(2^{-j_0\beta}+ N^{-1/2} 2^{j_0/2})),\\
\hat \Omega_{\ell,m}-\Omega_{\ell,m}&=\BigO_\P(\log(N)(2^{-j_0\beta}+ N^{-1/2} 2^{j_0/2})).
\end{align*}
\end{thm}
The condition in Theorem~\ref{prop:vitesse} is slightly different from Theorem 3 of \cite{Moulines08Whittle}. Our result presents an additional $\log(N)$ term due to technical simplifications in the proof. However, it may be suppressed by adding technical details. The same arguments apply for the $\log(N)$ term appearing in the rate of the convergence of the matrix.

The optimal rate is then expressed by balancing the two terms appearing in the bound above.

\begin{cor}
\label{prop:optimalite}
Assume that (W1)-(W5) and Condition (C) hold. Taking $2^{j_0}=N^{1/(1+2\beta)}$,
\[
\hat \bd-\bd^0=\BigO_\P(N^{-\beta/(1+2\beta)}).
\]
\end{cor}

This corresponds to the optimal rate \citep{Giraitis97rate}. Fourier Whittle estimators in \cite{Lobato99} and \cite{Shimotsu07} obtained the rate $m^{1/2}$ where $m$ is the number of discrete frequencies used in the Fourier transform. When $m\sim c N^\zeta$ with a positive constant $c$, the convergence is obtained for $0<\zeta<2\beta/(1+2\beta)$. Wavelet estimators thus give a slightly better rate of convergence.

Result of Corollary~\ref{prop:optimalite} stresses that it is necessary to fix the finest frequency $j_0$ in the wavelet procedure at a given scale depending on the regularity $\beta$ of the density $\bff^S(\cdot)$. A possible extension is to develop an estimation which is adaptive relatively to the parameter $\beta$. This is done {\it e.g.}~in univariate Fourier analysis by \cite{Iouditsky}. However, this topic exceeds the scope of this paper.

Further results on asymptotic normality, and in particular the asymptotic variance of the estimators, would give important information to quantify the quality of the estimators. In particular it would give a theoretical mean of comparison between the Fourier-based and the wavelet-based approaches or between the univariate and the multivariate estimations of $\bd$. This work is in progress and will be established in a future paper. Here, the comparison is done with a simulation study.

\section{Simulations}
\label{sec:simu}

In this section, simulated data are used to study the behaviour of the proposed procedure using one illustrative example. An extensive simulation study would exceed the scope of this paper, and will be provided in a future paper. Here, we consider an $ARFIMA(0,\bd,0)$ with a long-run correlation matrix $\bOmega=\begin{pmatrix}
1 & \rho \\ \rho & 1
\end{pmatrix}$ and $\rho =0.4$.
The proposed multivariate wavelet Whittle (MWW) estimators are computed for $N=512$ observations and 1000 Monte-Carlo replications. A R package named \textit{multiwave} is available and Matlab codes are available on request.

A set of different values of $\bd$ is considered. The choices are restricted to settings where the two components of the processes share the same order of stationarity. Indeed, it seems natural that time series measuring similar phenomena have similar stationary properties. 
We simulated both stationary and nonstationary ARFIMA processes, $d_1=0.2$ and $d_1=1.2$. Our MWW estimator is shown to be consistent and the quality increases when $|\rho|$ increases. We also conducted a comparison between our estimators and multivariate Fourier Whittle (MFW) estimators developed by \cite{Shimotsu07} for stationary processes only. In nonstationary cases, \cite{Nielsen11} proposed a similar approach based on the extended Fourier transform of \cite{Abadir}. However in our simulations, this approach gives satisfactory results only for $ \bd<1.5$. 

A wavelet-based procedure with $M$ vanishing moments should be compared with an estimation based on tapered Fourier of order $M$. Such a comparison has been driven in \cite{FayMoulinesRoueffTaqqu} in one-dimensional settings. The authors established that wavelet-based estimation outperforms tapered Fourier estimation. Similar observations are expected in multivariate framework. Yet as multivariate Whittle estimation based on tapered Fourier transform has not been studied in literature, we choose not to display such a comparison.

It is worth pointing out that the main advantage of wavelets is their flexibility. Wavelet-based estimators can be applied for a large set of data, whatever the degree of stationarity is (if still smaller than the number of vanishing moments) and even if the processes contain polynomial trends which is very attractive for real data applications.

{\bf Parameters used for estimation.}
   
The quality of estimation by wavelets relies on the choice of the wavelet bases. A trade-off is necessary between the number of vanishing moments and the support size of the wavelets. In time series analysis, the number of vanishing moments enables to consider polynomial trends or nonstationary time series, due to the constraint $\sup_\ell d_\ell\leq M$. Yet, the support size of the wavelet is proportional to the number of vanishing moments and increases the variance of estimation. 
   
The wavelet basis used in this section is the Daubechies wavelet with $M=4$ vanishing moments. Its regularization parameter is $\alpha=1.91$. In our framework, when considering stationary time series, we could also apply our procedure using Haar bases. Estimation based on Haar wavelet indeed gives better results, possibly better than Fourier (see {\it e.g.}  \cite{GencaySignori} in the case of tests of serial correlation). As explained above a lower number of vanishing moments improves the quality of the wavelet-based estimators, see~\cite{FayMoulinesRoueffTaqqu} in the univariate case. Similar results are observed in multivariate estimations. They are not presented here for the sake of concision. As our goal is to propose a flexible method for real data application, we prefer to consider a higher vanishing moments bases to stress its flexibility.

The method is controlled by the scales $j_0$ and $j_1$. The scale $j_1$ is fixed equal to $\log_2(N)$ while $j_0$ is chosen so that the optimal mean square error is minimal. Increasing $j_0$ leads to a smaller bias but a higher variance since less coefficients are used in the estimation process, which may be controlled by an adaptive procedure. As stated by Theorem~\ref{prop:convergence}, the finest scales have to be removed from estimation to get rid of the presence of the short-range dependence $\bff^S(\cdot)$. Similar considerations can be found in \cite{Achard08} and \cite{FayMoulinesRoueffTaqqu}.

Concerning MFW estimation, the main parameter is the number $m$ of frequencies used in the procedure. An usual choice in literature is $m=N^{0.65}$ (see {\it e.g.}~\cite{Shimotsu07} or \cite{Nielsen05survey}). Additionally MFW estimators are evaluated using values of $m$ giving the same number of Fourier coefficients than of wavelet coefficients. The final $m$ kept is the one giving the optimal mean square error. The parallel between the number of wavelet scales and the number of Fourier frequencies has been discussed in \cite{FayMoulinesRoueffTaqqu}.  

{\bf Measures of quality.}

The quality of the estimators is measured by the bias, the standard deviation (std) and the RMSE which is equal to the square root of $(bias^2+std^2)$. In order to display an easy comparison between the univariate and the multivariate approaches, we compute the ratio between the RMSE obtained with the multivariate wavelet Whittle estimation and the RMSE obtained with univariate wavelet Whittle estimations. It is denoted by {\it ratio M/U}. A similar quantity is defined for the comparison with MFW estimation. We define {\it ratio W/S} to be the ratio between the RMSE respectively using wavelet-based estimators and Fourier-based estimators.

\subsection[Estimation of the long-memory parameters]{Estimation of the long-memory parameters $\bd$}

Results for the estimation of $\bd$ are presented in Table~\ref{tab:dwav04}. The ratio M/U points out that the quality of estimation is increased with the multivariate approach with respect to the univariate procedure. When the series are correlated, it is better to use MWW estimators to infer the long-memory parameters. The estimation is still satisfactory in nonstationary settings. 

Table~\ref{tab:dfou04} displays the results of the MFW estimators described in \cite{Shimotsu07}. With the usual number of frequencies $m=N^{0.65}$ in Fourier-based estimation, our wavelet-based procedure leads to lower RMSEs, as quantified by the ratio W/F. More precisely the good performance of our scheme of estimation is due to a lower variance, even if the bias is higher. With a higher number of frequencies in Fourier-based estimation, taking a value that minimizes the RMSE, the MWW estimators are no more preferable to MFW. Yet, the ratio W/F stays close to 1 and the analysis of the bias and variances reveals similar orders of magnitude.

\begin{table}
\caption{Multivariate Whittle wavelet estimation of $\bd$ for a bivariate $ARFIMA(0,\bd,0)$ with $\rho=0.4$, $N=512$ with 1000 repetitions.}
\label{tab:dwav04}
\begin{center}
\begin{tabular}{p{0.5\textwidth}p{0.45\textwidth}}

\begin{tabular}{@{} l@{~~~}r@{~~~} c@{~~~}c@{~~~}c@{~~~}c@{}} 
\toprule
\multicolumn{6}{c}{$j_0=1$.}\\
 \midrule 
$d_1$ & $\bd$ & bias & std & RMSE & ratio M/U \\

\midrule \midrule

0.2  & 0.2 &  -0.0267  & 0.0413  & 0.0492 & 0.9080  \\ 
 & -0.2   &  0.0379  & 0.0430  & 0.0574 & 1.0595  \\  \hline  
 &  0.2   &  -0.0298  & 0.0428  & 0.0522 & 0.9631  \\ 
 &  0.0   &  -0.0002  & 0.0438  & 0.0438 & 0.9504  \\  \hline
 &  0.2   &  -0.0330  & 0.0456  & 0.0563 & 0.9713  \\ 
 &  0.2   &  -0.0333  & 0.0443  & 0.0554 & 0.9831  \\ \hline
 &  0.2   &  -0.0304  & 0.0429  & 0.0526 & 0.9583  \\ 
 &  0.4   &  -0.0571  & 0.0461  & 0.0734 & 0.9701  \\  
 \bottomrule

\end{tabular} &

\begin{tabular}{@{} l@{~~~}r@{~~~} c@{~~~}c@{~~~}c@{~~~}c@{}} 
\toprule
\multicolumn{6}{c}{$j_0=2$.}\\
 \midrule 
$d_1$ & $\bd$  & bias & std & RMSE & ratio M/U \\

\midrule \midrule

1.2 & 1.2  &  -0.0380  & 0.0830  & 0.0913 & 0.9728  \\ 
 & 0.8  &  -0.0298  & 0.0775  & 0.0831 & 0.9643  \\   \hline
 & 1.2  &  -0.0360  & 0.0818  & 0.0894 & 0.9702  \\ 
 & 1.0  &  -0.0346  & 0.0808  & 0.0879 & 0.9626  \\  \hline
 & 1.2  &  -0.0463  & 0.0853  & 0.0970 & 0.9677  \\ 
 & 1.2  &  -0.0393  & 0.0850  & 0.0936 & 0.9688  \\  \hline
 & 1.2  &  -0.0369  & 0.0799  & 0.0880 & 0.9589  \\ 
 & 1.4  &  -0.0482  & 0.0863  & 0.0989 & 0.9648  \\    

\bottomrule
\end{tabular}
\end{tabular}

\end{center}
\end{table}

\begin{table}
\caption{Multivariate Whittle Fourier estimation of $\bd$ for a bivariate $ARFIMA(0,\bd,0)$ with $\rho=0.4$, $N=512$ with 1000 repetitions. Two number of frequencies $m$ are presented: the usual choice $m=\lfloor N^{0.65}\rfloor$ and the value giving the lower RMSE. $\lfloor x\rfloor$ denotes the closest integer smaller than $x$.}
\label{tab:dfou04}
\begin{center}
\begin{tabular}{p{0.45\textwidth}p{0.45\textwidth}}

\begin{tabular}{@{} r@{~~~}c@{~~~}c@{~~~}c@{~~~}c@{}} 
\toprule
 \multicolumn{5}{c}{$m=\lfloor N^{0.65}\rfloor =57$.}\\
 \midrule 
  $\bd$ &  bias & std & RMSE & ratio W/F \\
\midrule
\midrule
   0.2 & -0.0087  & 0.0707  & 0.0712  & 0.6908 \\ 
  -0.2 & -0.0001  & 0.0824  & 0.0824  & 0.6958 \\  \hline
   0.2 & -0.0037  & 0.0679  & 0.0680  & 0.7674 \\ 
   0.0 & -0.0010  & 0.0778  & 0.0778  & 0.5630 \\ \hline 
   0.2 & -0.0078  & 0.0691  & 0.0695  & 0.8101 \\ 
   0.2 & -0.0043  & 0.0733  & 0.0735  & 0.7546 \\  \hline  
   0.2 & -0.0038  & 0.0705  & 0.0706  & 0.7445 \\ 
   0.4 &  0.0012  & 0.0788  & 0.0788  & 0.9320 \\ 
\bottomrule
\end{tabular} & 
\begin{tabular}{@{} r@{~~~}c@{~~~}c@{~~~}c@{~~~}c@{}} 
\toprule
 \multicolumn{5}{c}{$m=\lfloor N^{0.876}\rfloor = 236$.}\\
 \midrule
 $\bd$ &  bias & std & RMSE & ratio W/F \\
\midrule
\midrule
  0.2 & -0.0174  & 0.0318 & 0.0362 & 1.3581 \\ 
 -0.2 &  0.0158  & 0.0323 & 0.0359 & 1.5964 \\ \hline
  0.2 & -0.0170  & 0.0315 & 0.0358 & 1.4558 \\ 
  0.0 & -0.0025  & 0.0318 & 0.0319 & 1.3728 \\ \hline
  0.2 & -0.0200  & 0.0321 & 0.0378 & 1.4875 \\ 
  0.2 & -0.0189  & 0.0320 & 0.0372 & 1.4905 \\  \hline
  0.2 & -0.0201  & 0.0325 & 0.0382 & 1.3759 \\ 
  0.4 & -0.0317  & 0.0366 & 0.0484 & 1.5169 \\               
\bottomrule
\end{tabular}
\end{tabular}
\end{center}
\end{table}

\subsection[Estimation of the long-run covariance]{Estimation of the long-run covariance $\Omega$}

This section deals with the estimation of the long-run covariance matrix $\bOmega$ and the estimation of the correlation $\Omega_{12}/\sqrt{\Omega_{11}\Omega_{22}}.$  This latter quantity corresponds in literature to the power-law coherency between the two time series \citep{SelaHurvich2012} or to the fractal connectivity \citep{Achard08}. 

The results obtained in simulations for MWW estimation of the covariance and correlation are given in Table~\ref{tab:Owav04}. The quality is satisfactory in all settings, especially in the stationary ones. 

The results for MFW estimation are displayed in Table~\ref{tab:Ofou04}. When MFW is applied with $m=N^{0.65}$ frequencies, the ratio W/F is less than 1. Like for the estimation of $\bd$ the good performance of MWW estimators is principally due to a smaller variance. When MFW estimators are implemented with a higher number of frequencies, giving optimal results for the estimation of $\bd$, the difference between MWW and MFW procedures decreases. The quality of the two estimation schemes are similar, with comparable values for bias and variances.

\begin{table}
\caption{Wavelet Whittle estimation of $\bOmega$ for a bivariate $ARFIMA(0,\bd,0)$ with $\rho=0.4$, $N=512$ with 1000 repetitions.}
\label{tab:Owav04}
\begin{center}
\begin{tabular}{p{0.5\textwidth}p{0.45\textwidth}}

\begin{tabular}{@{} l@{~~~}c@{~~~}c@{~~~}c@{~~~}c@{~~~}c@{}} 
\toprule
\multicolumn{5}{c}{$j_0=1$.}\\
 \midrule 
 $\bd$ &   & bias & std & RMSE \\
\midrule \midrule
 (0.2,-0.2)& $\Omega_{1,1}$  &   0.0342  & 0.0710  & 0.0788 \\ 
          & $\Omega_{1,2}$    &   0.0387  & 0.0605  & 0.0718 \\ 
          & $\Omega_{2,2}$   &  -0.0402  & 0.0709  & 0.0815 \\ 
           & correlation       &   0.0400  & 0.0496  & 0.0637 \\  \hline
 (0.2,0.0) & $\Omega_{1,1}$  &   0.0309  & 0.0697  & 0.0762 \\ 
          & $\Omega_{1,2}$    &   0.0176  & 0.0540  & 0.0568 \\ 
           & $\Omega_{2,2}$    &  -0.0012  & 0.0732  & 0.0733 \\ 
           & correlation       &   0.0113  & 0.0417  & 0.0432 \\ \hline
(0.2,0.2) & $\Omega_{1,1}$  &   0.0297  & 0.0733  & 0.0790 \\ 
           & $\Omega_{1,2}$    &   0.0116  & 0.0518  & 0.0530 \\ 
           & $\Omega_{2,2}$    &   0.0282  & 0.0725  & 0.0778 \\ 
          & correlation      &  -0.0003  & 0.0386  & 0.0386 \\  \hline
 (0.2,0.4) & $\Omega_{1,1}$  &  0.0356  & 0.0703  & 0.0788 \\ 
          & $\Omega_{1,2}$   &  0.0328  & 0.0568  & 0.0655 \\ 
           & $\Omega_{2,2}$    &  0.0707  & 0.0728  & 0.1015 \\ 
          & correlation      &  0.0106  & 0.0422  & 0.0435 \\ 
\bottomrule
\end{tabular} &
\begin{tabular}{ l@{~~~}c@{~~~}c@{~~~}c@{~~~}c@{}}
\toprule
\multicolumn{5}{c}{$j_0=2$.}\\
 \midrule 
 $\bd$ &  & bias & std & RMSE \\
\midrule
\midrule
 (1.2,0.8) & $\Omega_{1,1}$  &  0.0037  & 0.1473  & 0.1474 \\ 
           & $\Omega_{1,2}$    &  0.0478  & 0.1199  & 0.1290 \\ 
           & $\Omega_{2,2}$    &  0.0052  & 0.1303  & 0.1304 \\ 
          & correlation      &  0.0462  & 0.1041  & 0.1139 \\  \hline
(1.2,1.0) & $\Omega_{1,1}$  &  -0.0031  & 0.1411  & 0.1411 \\ 
          & $\Omega_{1,2}$    &  0.0182  & 0.1003  & 0.1019 \\ 
          & $\Omega_{2,2}$    &  0.0027  & 0.1357  & 0.1357 \\ 
          & correlation       &  0.0176  & 0.0781  & 0.0800 \\  \hline
 (1.2,1.2) & $\Omega_{1,1}$  &  0.0055  & 0.1442  & 0.1443 \\ 
            & $\Omega_{1,2}$    &  0.0060  & 0.0921  & 0.0923 \\ 
           & $\Omega_{2,2}$    &  -0.0033  & 0.1456  & 0.1456 \\ 
           & correlation      &  0.0052  & 0.0685  & 0.0687 \\  \hline
(1.2,1.4) & $\Omega_{1,1}$  &  0.0001  & 0.1496  & 0.1496 \\ 
           & $\Omega_{1,2}$    &  0.0155  & 0.1039  & 0.1051 \\ 
           & $\Omega_{2,2}$    &  0.0135  & 0.1610  & 0.1615 \\ 
           & correlation      &  0.0125  & 0.0802  & 0.0812 \\ 
\bottomrule
\end{tabular}
\end{tabular}
\end{center}
\end{table}

\begin{table}
\caption{Fourier Whittle estimation of $\bOmega$ for a bivariate $ARFIMA(0,\bd,0)$ with $\rho=0.4$, $N=512$ with 1000 repetitions. Two number of frequencies $m$ are presented: the usual choice $m=\lfloor N^{0.65}\rfloor$ and the value giving the lower RMSE.}
\label{tab:Ofou04}
\begin{center}
\begin{tabular}{ll}

\begin{tabular}{@{} l@{~~~}c@{~~~}c@{~~~}c@{~~~}c@{~~~}c@{}} 
\toprule
 \multicolumn{6}{c}{$m=\lfloor N^{0.65}\rfloor = 57$.}\\
 \midrule
 $\bd$ &  & bias & std & RMSE & ratio W/F \\
\midrule \midrule

  (0.2,-0.2)& $\Omega_{1,1}$  &  0.0394  & 0.2253   & 0.2287  & 0.3444  \\ 
            & $\Omega_{1,2}$  &  0.0091  & 0.1156   & 0.1160  & 0.6189  \\ 
            & $\Omega_{2,2}$  &  0.0145  & 0.2308   & 0.2313  & 0.3525  \\ 
            & correlation     & -0.0002  & 0.0774   & 0.0774  & 0.8229  \\ \hline

  (0.2,0.0) & $\Omega_{1,1}$  & 0.0245  & 0.2245   & 0.2259  & 0.3373 \\ 
            & $\Omega_{1,2}$  & 0.0124  & 0.1154   & 0.1161  & 0.4892 \\ 
            & $\Omega_{2,2}$  & 0.0163  & 0.2341   & 0.2347  & 0.3121 \\ 
            & correlation     & 0.0061  & 0.0793   & 0.0795  & 0.5428 \\  \hline

  (0.2,0.2) & $\Omega_{1,1}$  & 0.0319  & 0.2319   & 0.2341  & 0.3376 \\ 
            & $\Omega_{1,2}$  & 0.0141  & 0.1191   & 0.1199  & 0.4423 \\ 
            & $\Omega_{2,2}$  & 0.0236  & 0.2331   & 0.2343  & 0.3321  \\ 
            & correlation     & 0.0041  & 0.0781   & 0.0782  & 0.4935  \\ \hline

  (0.2,0.4) & $\Omega_{1,1}$  & 0.0264  & 0.2255   & 0.2271  & 0.3470 \\ 
            & $\Omega_{1,2}$  & 0.0107  & 0.1232   & 0.1237  & 0.5298 \\ 
            & $\Omega_{2,2}$  & 0.0276  & 0.2462   & 0.2478  & 0.4096  \\ 
            & correlation     & 0.0001  & 0.0783   & 0.0783  & 0.5548  \\
            \bottomrule
\end{tabular} &  
\begin{tabular}{c@{~~~}c@{~~~}c@{~~~}c@{~~~}c@{}} 
\toprule
 \multicolumn{4}{c}{$m=\lfloor N^{0.876} \rfloor = 236$.}\\
 \midrule
   bias & std & RMSE & ratio W/F \\
\midrule \midrule 

    0.0492  & 0.0679   & 0.0839  & 0.9395  \\ 
    0.0009  & 0.0498   & 0.0498  & 1.4414  \\ 
   -0.0470  & 0.0640   & 0.0794  & 1.0273  \\ 
    0.0006  & 0.0387   & 0.0387  & 1.6464  \\ \hline

    0.0449  & 0.0666   & 0.0803  & 0.9486 \\ 
    0.0105  & 0.0506   & 0.0517  & 1.0985 \\ 
    -0.0008  & 0.0677   & 0.0677  & 1.0819 \\ 
      0.0014  & 0.0383   & 0.0383  & 1.1259 \\  \hline

     0.0450  & 0.0708   & 0.0839  & 0.9417 \\ 
     0.0176  & 0.0520   & 0.0549  & 0.9666 \\ 
      0.0438  & 0.0690   & 0.0818  & 0.9517  \\ 
   -0.0006  & 0.0382   & 0.0382  & 1.0099  \\ \hline

   0.0489  & 0.0682   & 0.0839  & 0.9392 \\ 
    0.0313  & 0.0531   & 0.0616  & 1.0632 \\ 
    0.1052  & 0.0705   & 0.1267  & 0.8012  \\ 
    0.0002  & 0.0384   & 0.0384  & 1.1307  \\

\bottomrule
\end{tabular}
\end{tabular}
\end{center}
\end{table}

To conclude, the multivariate approach increases the quality of estimation of the long-memory parameters $\bd$ in comparison with a univariate estimation. In stationary frameworks, the performance is very similar to multivariate Fourier Whittle estimation, when estimating the vector $\bd$ or the long-run covariance matrix. The main advantage of our wavelet-based procedure is then its flexibility. By contrast with Fourier-based estimation, our estimators can be applied in a larger scope of situations, with nonstationary processes or in the presence of polynomial trends in the time series.

\section{Application on neuroscience data}

\label{sec:real}

We apply our approach to neuroscience data where the recorded data are typical example of multivariate long-range dependent time series. Researchers are interested in characterising the brain connectivity. Usually, the connectivity is evaluated using correlations at different frequencies between time series measuring the brain activity. We will show in this section that our method is perfectly adequate to deal with these real data. 

The study concerns MEG data acquired from a healthy 43 year old woman studied during rest with eyes open at the National Institute of Mental Health Bethesda, MD using a 274-channel CTF MEG system VSM MedTech, Coquitlam,
BC, Canada operating at 600 Hz. The data were previously used in \cite{Achard08}. We consider $N=2^{15}$ time points for each of the 274 time series.

Figure~\ref{fig:meg.ex} displays the time series for arbitrary four channels. It is clear that they present nonlinear trends. Consequently, Fourier methods are not adequate to analyse such data, and methods based on wavelets are better to use.

\begin{figure}[!ht]
\caption{MEG recordings for 4 arbitrary channels.}
\label{fig:meg.ex}
\begin{center}
\includegraphics[width=10cm,height=5cm]{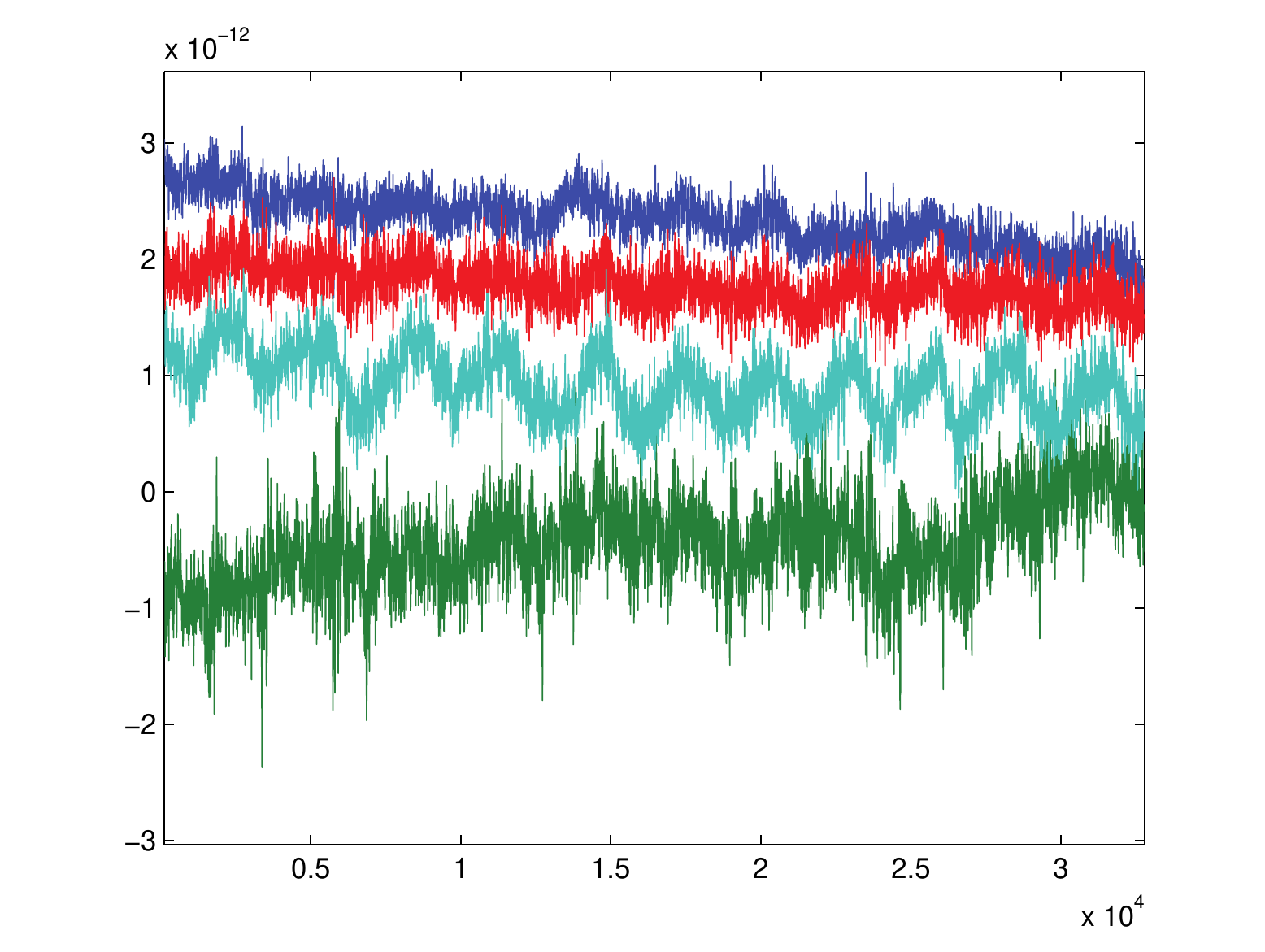}
\end{center}
\end{figure}

Our procedure was applied using scales 4 to 8. It corresponds to frequencies between 1 to 20 Hz. This choice was motivated by discussions with neuroscientists. It takes into account the presence of high-frequency noise which is modelled by $\bff^S(\cdot)$. The data were preprocessed, and the low-frequencies were removed. Figure~\ref{fig:meg.result} presents the results of the estimation of the long-memory parameters $\bd$ and of the long-run covariance matrix $\bOmega$.

\begin{figure}[!ht]
\caption{Results obtained by MWW estimators on the MEG dataset: histogram of the estimated long-memory parameters $\bd$ (a) and estimated fractal connectivity matrix (b).}
\label{fig:meg.result}
\begin{center}
\subfigure[]{\includegraphics[scale=0.4]{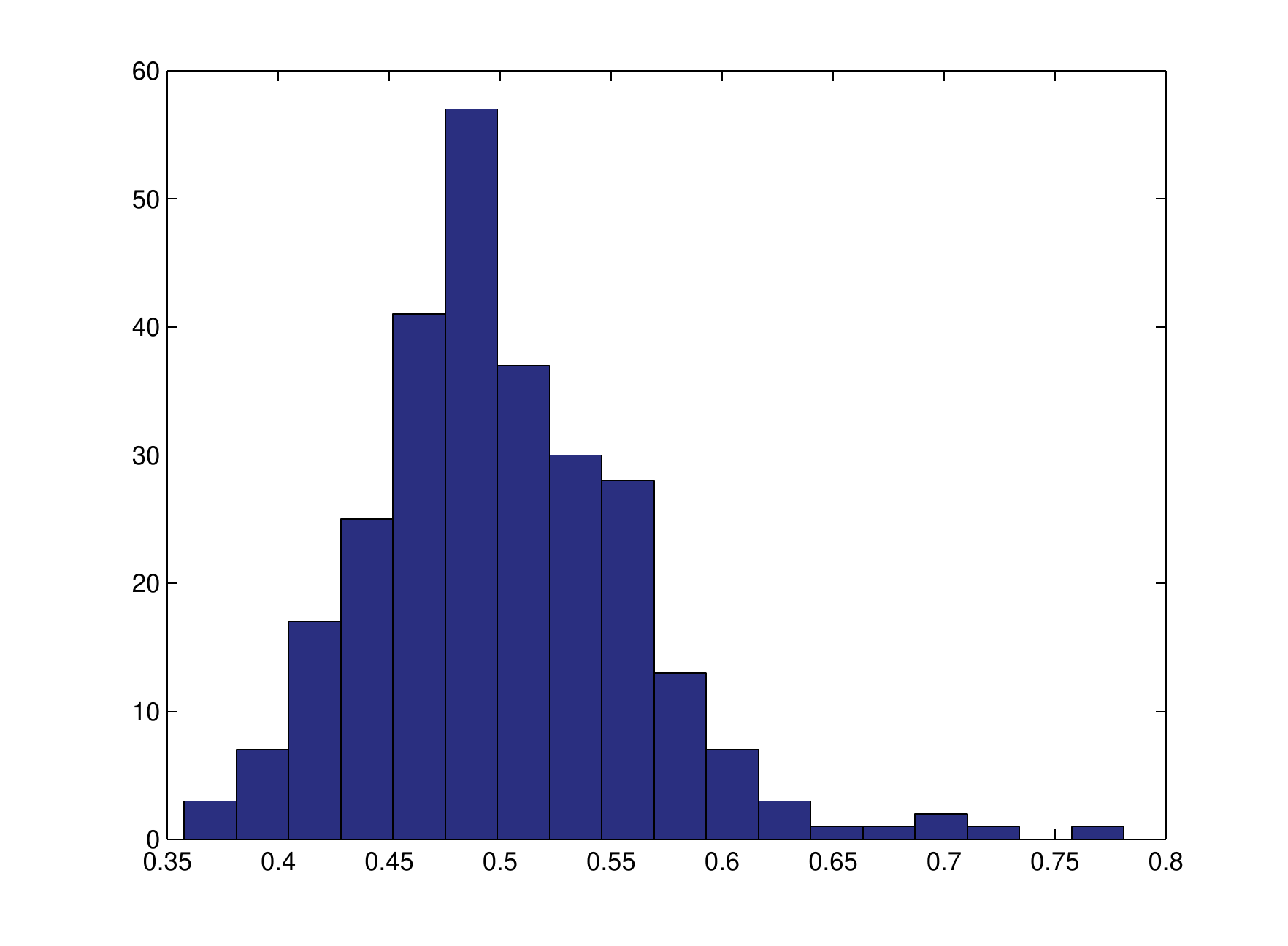}}
\subfigure[]{\includegraphics[scale=0.4]{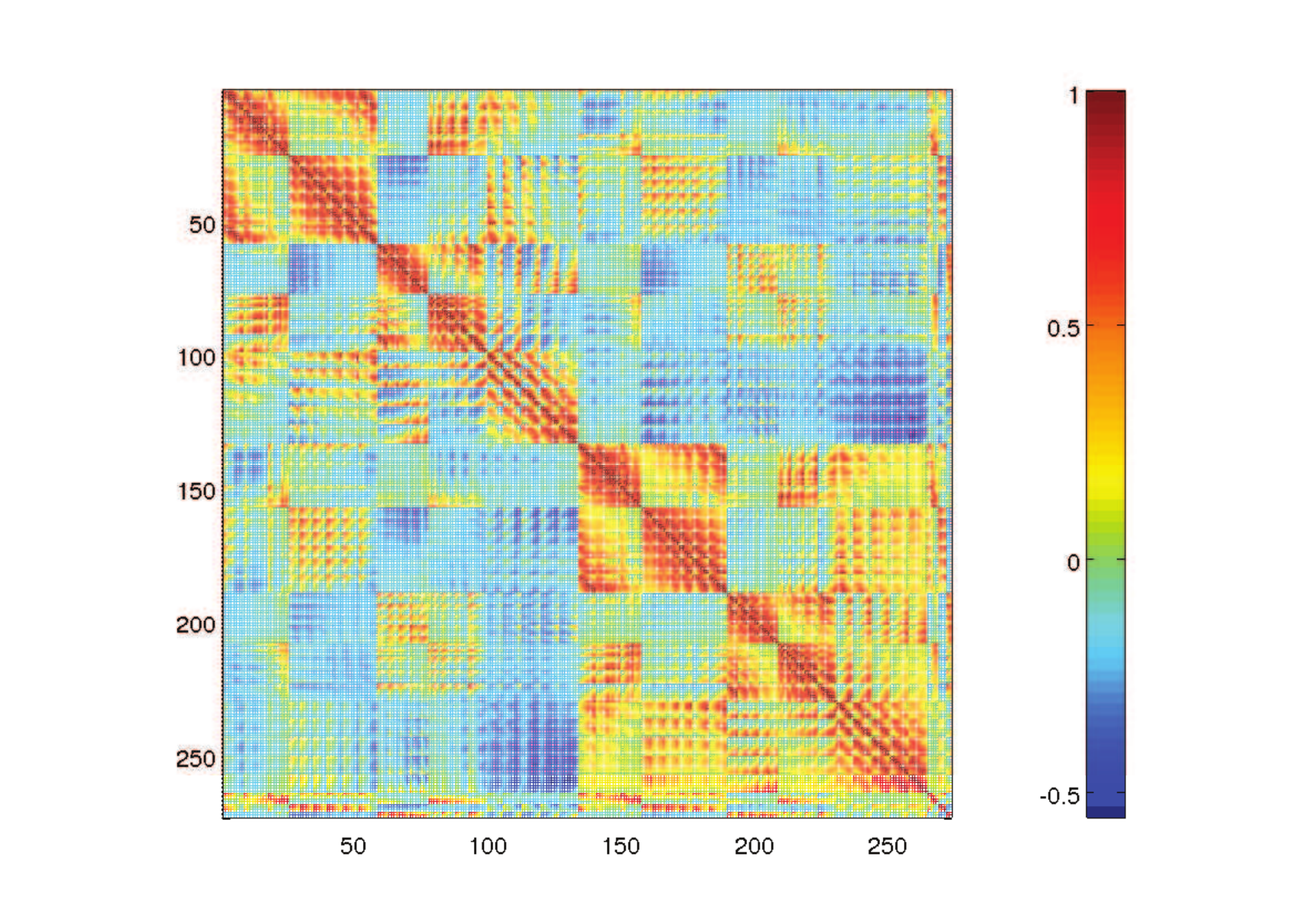}}
\end{center}
\end{figure}

First, the histogram of the estimate $\hat \bd$ shows that the maximal difference between the values of the long-memory parameters is less than 0.5, and the problem of identifiability of $\bOmega$ does not occur with these data. This allows us to give an estimate of the fractal connectivity. It is worth noticing that clusters appear in the correlation matrix. Most of them are situated along the diagonal, corresponding to spatially closed channels. Some channels are still correlated, even far from each others. It would be interesting to relate this result to a neuroscience interpretation. This will be investigated in future work.

\section*{Conclusion}

Many application fields are concerned with high-dimensional time series. A challenge is to characterize their long-memory properties and their correlation structure. The present work consider a semiparametric multivariate model, including a large class of multivariate processes such as some fractionally integrated processes. We propose an estimation of the long-dependence parameters and of the fractal connectivity, based on the Whittle approximation and on a wavelet representation of the time series. The theoretical properties of the estimation show the asymptotic optimality. A simulation study confirms that the estimation is well-behaved on finite samples. Finally we propose an application to the estimation of a human brain functional network based on MEG data sets. Our study highlights the benefit of the multivariate analysis, namely improved efficiency of estimation of dependence parameters and estimation of long-term correlations. Future work may concern the asymptotic normality of the estimators, since the development of tests may present a significant benefit for real data applications.

\vspace{10pt}

{\bf Acknowledgements.} The authors thank the Co-Editor and two anonymous referees for their comments that
led to substantial improvements in the paper. This work was partly supported by the IXXI research institute and the ANR project GRAPHSIP ANR-14-CE27-0001. This work has been done during the provisional assignment of I. Gannaz within CNRS at GIPSA-lab, Grenoble. The authors are grateful to Shimotsu (\url{http://shimotsu.web.fc2.com/Site/Matlab\_Codes.html}) and to Fa\"y, Moulines, Roueff and Taqqu for kindly providing the codes of their respective papers.

\begin{appendices}

\section{Proof of Propositions~\ref{prop:cov_ondelettes0},~\ref{prop:cov_ondelettes} and \ref{prop:cov_croisee}}

This section deals with the proof of Propositions~\ref{prop:cov_ondelettes0} and \ref{prop:cov_ondelettes}. The proof of Proposition~\ref{prop:cov_croisee} is based on similar arguments and is omitted.

The covariance between $W_{j,k}(\ell)$ and $W_{j,k}(m)$ can be written with the cospectrum, $\theta_{\ell,m}(j)=\int_\R {Re(f_{\ell,m}(\lambda))|\H_j(\lambda)|^2\,d\lambda }$. Indeed as the cross-spectral density is Hermitian, its imaginary part is an odd function,
$$\theta_{\ell,m}(j)=\Omega_{\ell,m}\int_\R {|2\sin({\lambda}/{2})|^{-(d_\ell+d_m)}\cos((\pi \text{sign}(\lambda)-\lambda)(d_\ell-d_m)/2)f^S_{\ell,m}(\lambda)|\H_j(\lambda)|^2\,d\lambda }.$$
The sinus function being odd,
$$\theta_{\ell,m}(j)=\Omega_{\ell,m}\cos(\pi(d_\ell-d_m)/2)\int_\R {|2\sin({\lambda}/{2})|^{-(d_\ell+d_m)}\cos(\lambda(d_\ell-d_m)/2)f^S_{\ell,m}(\lambda)|\H_j(\lambda)|^2\,d\lambda }.$$

The proof is very similar to Theorem 1 of \cite{Moulines07SpectralDensity}. Define the quantities $A_{\ell,m}(j)$ and $R_{\ell,m}(j)$,
\begin{align*}
A_{\ell,m}(j)&= \Omega_{\ell,m} 2^{j} \cos(\pi(d_\ell-d_m)/2) \int_{-\pi}^{\pi}{|2\sin({\lambda}/{2})|^{-(d_\ell+d_m)}\cos(\lambda(d_\ell-d_m)/2)} \\
&  \mbox{}\hspace{9cm}  {f_{\ell,m}^S(\lambda)|\hat\phi(\lambda)\hat\psi(2^{j}\lambda)|^2\,d\lambda}\\
R_{\ell,m}(j)&=\theta_{\ell,m}(j)-A_{\ell,m}(j)
\end{align*}
Following the proof of \cite{Moulines07SpectralDensity}, we can rewrite $A_{\ell,m}(j)$,
\begin{multline*}A_{\ell,m}(j) = \Omega_{\ell,m}2^{j} \cos(\pi(d_\ell-d_m)/2)\\
\int_{-\pi}^{\pi}{g_{\ell,m}(\lambda)|\lambda|^{-(d_\ell+d_m)}\cos(\lambda(d_\ell-d_m)/2)f_{\ell,m}^S(\lambda)|\hat\phi(\lambda)|^2|\hat\psi(2^{j}\lambda)|^2\,d\lambda}
\end{multline*}
\begin{equation*}
\text{with~} g_{\ell,m}(\lambda)=\left|\frac{2\sin({\lambda}/{2})}{\lambda}\right|^{-(d_\ell+d_m)}\text{~for all~}\lambda\in(-\pi,\pi).
\end{equation*}
\begin{itemize}
\item The assumption $\bff^S(\cdot)\in\mathcal H(\beta,L)$ states that $\left|{f_{\ell,m}^S(\lambda)}-1\right|\leq L |\lambda|^\beta$ for all $\lambda\in(-\pi,\pi)$.
\item Under assumption (W1) the function $|\hat\phi(\cdot)|^2$ is infinitely differentiable and bounded on $(-\pi,\pi)$.
\item Using a Taylor expansion, the function $\bg(\cdot)$ belongs to $\mathcal H(2,L_g)$ with $L_g=\sup_{\ell,m=1,\dots,p}\sup_{\lambda\in(-\pi,\pi)} |g_{\ell,m}"(\lambda)|$ where $\bg"(\cdot)$ denotes the second derivative of $\bg(\cdot)$.
\end{itemize}
This implies that there exists a constant $C_{\phi,d}$ depending on $\phi(\cdot)$ and $\bd$ such that
\begin{multline*}
\left|A_{\ell,m}(j) - \Omega_{\ell,m}2^{j} \cos({\pi(d_\ell-d_m)/2})\int_{-\pi}^{\pi}{|\lambda|^{-(d_\ell+d_m)}\cos({\lambda(d_\ell-d_m)/2})|\hat\psi(2^{j}\lambda)|^2\,d\lambda}\right|\\
\leq C_{\phi,d} L 2^{j}\int_{-\pi}^{\pi}|\lambda|^{(\beta-d_\ell-d_m)}|\hat\psi(2^{j}\lambda)|^2d\lambda.
\end{multline*}
With a change of variable,
\begin{multline*}
\left|A_{\ell,m}(j) - \Omega_{\ell,m}2^{j(d_\ell+d_m)} \cos(\pi(d_\ell-d_m)/2)\int_{-2^j\pi}^{2^j\pi}{|\lambda|^{-(d_\ell+d_m)}\cos(2^{-j}{\lambda(d_\ell-d_m)/2})|\hat\psi(\lambda)|^2\,d\lambda}\right|\\
\leq C_{\phi,d}L 2^{j(d_\ell+d_m-\beta)}\int_{-2^j\pi}^{-2^j\pi}|\lambda|^{(\beta-d_\ell-d_m)}|\hat\psi(\lambda)|^2d\lambda.
\end{multline*}  
Under assumptions (W2), there exists a positive constant $C_\psi$ such that $\int_{-\pi}^{\pi}|\lambda|^{(\beta-d_\ell-d_m)}|\hat\psi(\lambda)|^2d\lambda\leq C_\psi \int_{-\infty}^{\infty}|\lambda|^{(\beta+2\alpha)-d_\ell-d_m}d\lambda$. (W5) states that $(\beta+2\alpha)-d_\ell-d_m<1$, the right-hand side of the inequality is bounded by a constant depending on $\psi(\cdot)$, $\beta$ and $\min_{i=1,\dots,p} d_i$. Using assumptions (W2) and (W5), we also have $\left|\int_{|\lambda|>2^j\pi}{|\lambda|^{-(d_\ell+d_m)}\cos(2^{-j}{\lambda(d_\ell-d_m)/2})|\hat\psi(\lambda)|^2\,d\lambda}\right|\leq C_\psi \int_{|\lambda|>2^j\pi}{|\lambda|^{-(1+\beta)}\,d\lambda}.$ The right-hand side is bounded by a constant depending on $\psi$ and $\beta$. 
We obtain that there exists a constant $C_0$ depending on $\alpha$, $\beta$, $\phi(\cdot)$, $\psi(\cdot)$, $\min_{i=1,\dots,p} d_i$ and $\Omega_{\ell,m}$ such that
\begin{equation*}
\left|A_{\ell,m}(j) - \Omega_{\ell,m}2^{j(d_\ell+d_m)} \cos({\pi(d_\ell-d_m)/2})K_j(d_\ell+d_m)\right|
\leq C_0 L 2^{j(d_\ell+d_m-\beta)},
\end{equation*}
with $K_j(d_\ell,d_m)=\int_{-\infty}^\infty{|\lambda|^{-(d_\ell+d_m)}\cos(2^{-j}{\lambda(d_\ell-d_m)/2})|\hat\psi(\lambda)|^2\,d\lambda}$.

On the other hand, we can consider a first order approximation. 
Let 
\begin{equation*}A_{\ell,m}(j) = \Omega_{\ell,m}2^{j} \cos(\pi(d_\ell-d_m)/2)\\
\int_{-\pi}^{\pi}{g_{\ell,m}(\lambda)|\lambda|^{-(d_\ell+d_m)}f_{\ell,m}^S(\lambda)|\hat\phi(\lambda)|^2|\hat\psi(2^{j}\lambda)|^2\,d\lambda}
\end{equation*}
\begin{equation*}
\text{with now~} g_{\ell,m}(\lambda)=\left|\frac{2\sin({\lambda}/{2})}{\lambda}\right|^{-(d_\ell+d_m)}\cos(\lambda(d_\ell-d_m)/2)\text{~for all~}\lambda\in(-\pi,\pi).
\end{equation*}
Then a similar approximation is obtained,
\begin{multline*}
\left|A_{\ell,m}(j) - \Omega_{\ell,m}2^{j(d_\ell+d_m)} \cos(\pi(d_\ell-d_m)/2)\int_{-2^j\pi}^{2^j\pi}{|\lambda|^{-(d_\ell+d_m)}|\hat\psi(\lambda)|^2\,d\lambda}\right|\\
\leq C_{\phi,d} L 2^{j(d_\ell+d_m-\beta)}\int_{-2^j\pi}^{2^j\pi}|\lambda|^{(\beta-d_\ell-d_m)}|\hat\psi(\lambda)|^2d\lambda.
\end{multline*}
Using assumptions (W2) and (W5),
\begin{equation*}
\left|A_{\ell,m}(j) - \Omega_{\ell,m}2^{j(d_\ell+d_m)} \cos(\pi(d_\ell-d_m)/2)K(\delta)\right|\leq CL 2^{j(d_\ell+d_m-\beta)},
\end{equation*}
with $K(\delta)=\int_{-\infty}^{\infty}{|\lambda|^{-(\delta)}|\hat\psi(\lambda)|^2\,d\lambda}$. 

Finally, $R_{\ell,m}(j)$ is bounded by $R_{\ell,m}(j)\leq C L 2^{(d_\ell+d_m-\beta)j}$. This inequality follows from the approximation of the squared gain function of the wavelet filter given in Proposition~3 of \cite{Moulines07SpectralDensity} and from similar arguments to those given for $A_{\ell,m}(j)$. We do not detail the proof here for the sake of concision and we refer to the proof of Theorem~1 in \cite{Moulines07SpectralDensity}.

\section{Proof of Proposition~\ref{prop:condition}}
\label{proof:conditionC}

Since the wavelet $\psi$ admits $M$ vanishing moments, at each scale $j\geq 0$, the associated filter $\mathbb{H}_j$ is factorised as $\mathbb{H}_j(\lambda)=(1-e^{i\lambda})^M \tilde{\mathbb{H}}_j(\lambda)$, with $\tilde{\mathbb{H}}_j$ trigonometric polynomial, $\tilde{\mathbb{H}}_j(\lambda)=\sum_{t\in\Z}\tilde h_{j,t}e^{it\lambda}$.

Since $M\geq  D$, the wavelet coefficients may be written as $$W_{j,k}(\ell)=\sum_{t\in\Z} \tilde h_{j,2^jk-t}(\Delta^{D} X_\ell)(t)=\sum_{t\in\Z} \bB_\ell(j,2^jk-t)\bepsilon(t),$$
where $\bB_\ell(j,2^jk-t)= \tilde h_{j,2^jk-t}(\Delta^{M-D}\bA_\ell)(t)$. For all $\ell=1,\ldots,p$, the sequence $\{\bB_\ell(j,u)\}_{u\in\Z}$ belongs to $\ell^2(\Z)$.

We first give a preliminary result on the second order moment of $W_{j,k}(\ell)$,
$$
\E[W_{j,k}(\ell)^2]
 =  \sum_{t,t'\in\Z} \sum_{a,b=1,\dots p} B_{\ell,a}(j,2^j k-t)B_{\ell,b}(j,2^j k-t') \E[\epsilon_a(t)\epsilon_b(t')].
$$
Using the second-order properties of the process $\bepsilon$, the variance is equal to  
\begin{equation}\label{eqn:ordre2}
\E[W_{j,k}(\ell)^2]= \sum_{t\in\Z} \sum_{a=1,\dots p} B_{\ell,a}(j,2^j k-t)^2.
\end{equation}

Consider now $\E[I_{\ell,m}(j)^2]$,
\begin{align*}
\E[I_{\ell,m}(j)^2]&=\E\left[\left(\sum_k W_{j,k}(\ell)W_{j,k}(m)\right)^2\right]\\
&= \sum_{k,k'} \sum_{t,t',t",t"'\in\Z} \sum_{a,b,c,d=1,\dots p} B_{\ell,a}(j,2^j k-t)B_{m,b}(j,2^j k-t')B_{\ell,c}(j,2^j k'-t")B_{m,d}(j,2^j k'-t''')\\
&  \hspace{8cm} \E[\epsilon_a(t)\epsilon_b(t')\epsilon_c(t")\epsilon_d(t''')].
\end{align*}
The fourth order behaviour of $\bepsilon$ implies that 
\begin{align*}
\E[I_{\ell,m}(j)^2] =& \sum_{k,k'} \sum_{t\in\Z} \sum_{a,b,c,d=1,\dots p} \mu_{a,b,c,d} B_{\ell,a}(j,2^j k-t)B_{m,b}(j,2^j k+t)B_{\ell,c}(j,2^j k'-t)B_{m,d}(j,2^j k'-t)\\
& + \sum_{k,k'} \sum_{t\neq t'} \sum_{a,b=1,\dots p}  B_{\ell,a}(j,2^j k-t)B_{m,a}(j,2^j k-t)B_{\ell,b}(j,2^j k'-t')B_{m,b}(j,2^j k'-t')\\
&  +\sum_{k,k'} \sum_{t\neq t'} \sum_{a,b=1,\dots p}  B_{\ell,a}(j,2^j k-t)B_{m,b}(j,2^j k-t')B_{\ell,a}(j,2^j k'-t)B_{m,b}(j,2^j k'-t')\\
&  +\sum_{k,k'} \sum_{t\neq t'} \sum_{a,b=1,\dots p}  B_{\ell,a}(j,2^j k-t)B_{m,b}(j,2^j k-t')B_{\ell,b}(j,2^j k'-t')B_{m,a}(j,2^j k'-t).
\end{align*}
As $\E[I_{\ell,m}(j)]^2=\sum_{k,k'}\sum_{t,t'}\sum_{a,b} B_{\ell,a}(j,2^j k-t)B_{m,a}(j,2^j k-t)B_{\ell,b}(j,2^j k'-t')B_{m,b}(j,2^j k'-t')$, the variance of the scalogramm satisfies,
\begin{align*}
\lefteqn{Var(I_{\ell,m}(j))}\\
=& \sum_{k,k'} \sum_{t\in\Z} \sum_{a,b,c,d=1,\dots p} \mu_{a,b,c,d} B_{\ell,a}(j,2^j k-t)B_{m,b}(j,2^j k-t)B_{\ell,c}(j,2^j k'-t)B_{m,d}(j,2^j k'-t)\\
 & +\sum_{k,k'} \E[W_{j,k}(\ell)W_{j,k'}(\ell)]\E[W_{j,k}(m)W_{j,k'}(m)] + \sum_{k,k'} \E[W_{j,k}(\ell)W_{j,k'}(m)]\E[W_{j,k}(m)W_{j,k'}(\ell)]\\
 &- \sum_{k,k'} \sum_{t} \sum_{a,b=1,\dots p}  B_{\ell,a}(j,2^j k-t)B_{m,a}(j,2^j k-t)B_{\ell,b}(j,2^j k'-t)B_{m,b}(j,2^j k'-t)\\
 &- \sum_{k,k'} \sum_{t} \sum_{a,b=1,\dots p}  B_{\ell,a}(j,2^j k-t)B_{m,b}(j,2^j k-t)B_{\ell,a}(j,2^j k'-t)B_{m,b}(j,2^j k'-t)\\
 & -\sum_{k,k'} \sum_{t} \sum_{a,b=1,\dots p}  B_{\ell,a}(j,2^j k-t)B_{m,b}(j,2^j k-t)B_{\ell,b}(j,2^j k'-t)B_{m,a}(j,2^j k'-t).
\end{align*}
Finally, $Var(I_{\ell,m}(j))\leq  V_1+V_2+V_3$ with
\begin{align*}
V_1&=  |\sum_{k, k'} \E[W_{j,k}(\ell)W_{j,k'}(\ell)]\E[W_{j,k}(m)W_{j,k'}(m)]|,\\ 
V_2 &=  | \sum_{k, k'} \E[W_{j,k}(\ell)W_{j,k'}(m)]\E[W_{j,k}(m)W_{j,k'}(\ell)]|,\\
V_3&= (1+\mu_\infty)
\sum_{k, k'} \sum_{t\in\Z} \sum_{a,b,c,d=1,\dots p}  |B_{\ell,a}(j,2^j k-t)B_{m,b}(j,2^j k-t)B_{\ell,c}(j,2^j k'-t)B_{m,d}(j,2^j k'-t)|.
\end{align*}

\subsection*{Bounds $V_1$ and $V_2$}

Proposition~\ref{prop:cov_croisee} states that $Cov(W_{j,k}(\ell),W_{j,k'}(\ell))=\int_{-\pi}^\pi e^{i\lambda (k-k')}D_{0;0}^{(j)}(\lambda;(\ell,\ell))d\lambda.$
The quantity $\int_{-\pi}^\pi e^{i\lambda v }D_{0;0}^{(j)}(\lambda;(\ell,\ell))d\lambda$ is the $v$-th Fourier coefficient of the function $D_{0,0}^{(j)}(\cdot;2d_\ell)$. Consequently, Parseval theorem implies that $\sum_{v\in\Z} \E[W_{j,k}(\ell)W_{j,k+v}(\ell)]\E[W_{j,k}(m)W_{j,k+v}(m)]$ converges to $I_0^{(j)}(2d_\ell,2d_m)=\int_{-\pi}^\pi D_{0,0}^{(j)}(\lambda;(\ell,\ell))D_{0,0}^{(j)}(\lambda;(m,m)) d\lambda$. The approximation given in Proposition~\ref{prop:cov_croisee} yields $D_{0,0}^{(j)}(\lambda;(\ell,\ell))\leq \Omega_{\ell,\ell}2^{j2d_\ell}\tilde D_{0,0}(\lambda;2d_\ell)+ CL\pi2^{(2d_\ell-\beta)j}.$ 
Then, using Minkowski inequality
$$\frac{1}{n_j 2^{2j(d_\ell+d_m)}} {V_2}\leq 2 (\Omega_{\ell,\ell}^2 \tilde I_0 (2d_\ell)+C^2L^2\pi^22^{-2\beta j})^{1/2}(\Omega_{m,m}^2 \tilde I_0 (2d_m)+C^2L^2\pi^22^{-2\beta j})^{1/2}.$$
 where $\tilde I_0(\delta)=\int_{-\pi}^{\pi} \tilde D_{0,0}(\lambda;\delta)^2d\lambda$.
It follows that $\frac{1}{n_j 2^{2j(d_\ell+d_m)}} {V_1}$ is bounded by a constant independent of $j$ and depending only on $\bd$, $\bOmega$, $\beta$, $\phi(\cdot)$ and $\psi(\cdot)$.

Similar arguments apply to $V_2$. 
Therefore $\frac{1}{n_j 2^{2j(d_\ell+d_m)}} {V_2}$ is bounded by $\frac{\int_{-\pi}^\pi \tilde D_{0,0}(\lambda;d_\ell+d_m)^2 d\lambda}{2^{2j(d_\ell+d_m)}}.$ By Proposition~\ref{prop:cov_croisee}, $\frac{1}{n_j 2^{2j(d_\ell+d_m)}} {V_2}$ is bounded by a constant depending only on $\bd$, $\bOmega$, $\beta$, $\phi(\cdot)$ and $\psi(\cdot)$.

\subsection*{Bound $V_3$}
The quantity $V_3$ is equal to
$$(1+\mu_\infty)\sum_{k} \sum_{\substack{t\in\Z\\ t'\in \{t+2^j(k-k'), k'\in\Z\}}} \sum_{a,b,c,d=1,\dots p}\!  |B_{\ell,a}(j,2^j k-t)B_{m,b}(j,2^j k-t)B_{\ell,c}(j,2^j k-t')B_{m,d}(j,2^j k-t')|.
$$
Applying Minkowski inequality on $V_3$,
\begin{align*}
V_3\leq& (1+\mu_{\infty}) \sum_{k} \sum_{t\in\Z} \sum_{a,b=1,\dots p} |B_{\ell,a}(j,2^j k-t)B_{m,b}(j,2^j k-t)| \\  
& ~~~~~~~~~~~\left(\sum_{t'\in\{t+2^jk', k'\in\Z\}} \sum_{c,d=1,\dots p}B_{\ell,c}(j,2^j k-t')^2\right)^{1/2} \left(\sum_{t'\in\{t-2^jk', k'\in\Z\}} \sum_{c,d=1,\dots p}B_{m,d}(j,2^j k-t')^2\right)^{1/2}.
\end{align*}
Hence,
\begin{align*}
V_3\leq &(1+\mu_{\infty}) \sum_{k} \left(\sum_{t\in\Z} \sum_{a,b=1,\dots p}B_{\ell,a}(j,2^j k-t)^2\right)^{1/2} \left(\sum_{t\in\Z} \sum_{a,b=1,\dots p}B_{m,b}(j,2^j k-t)^2\right)^{1/2} \\ &  ~~~~~~~~~~~\left(\sum_{t'\in\Z} \sum_{c,d=1,\dots p}B_{\ell,c}(j,2^j k-t')^2\right)^{1/2} \left(\sum_{t'\in\Z} \sum_{c,d=1,\dots p}B_{m,d}(j,2^j k-t')^2\right)^{1/2}.
\end{align*}
Together with~\eqref{eqn:ordre2} the following inequality is obtained, $ V_3\leq (1+\mu_{\infty}) n_j p^2 \theta_{\ell,\ell}(j)\theta_{m,m}(j).$
To conclude $$\frac{1}{n_j 2^{2j(d_\ell+d_m)}} {V_1}\leq(1+ \mu_\infty) p^2 \frac{\theta_{\ell,\ell}(j)\theta_{m,m}(j)}{2^{2j(d_\ell+d_m)}}.$$
Condition (C) is proved since  $\frac{\theta_{\ell,\ell}(j)\theta_{m,m}(j)}{2^{2j(d_\ell+d_m)}}$ tends to a constant independent of $j$ thanks to Proposition~\ref{prop:cov_ondelettes}. 

\section{Preliminary results}

\label{sec:S}

Let us take $\ell$ and $m$ in $1,\dots,p$, and define, for any sequence $\mu=\{\mu_{j},\,j\geq 0\}$, 
\begin{equation}
\label{eqn:Smu}
S_{\ell,m}(\mu)=\sum_{j,k} \mu_{j} \left(\frac{W_{j,k}(\ell)W_{j,k}(m)}{ 2^{j(d_\ell^0+d_m^0)}}-G_{\ell,m}^0\right)=\sum_{j=j_0}^{j_1} \mu_{j} \left(\frac{I_{\ell,m}(j)}{2^{j(d_\ell^0+d_m^0)}}-n_j G_{\ell,m}^0\right).
\end{equation}

\begin{prop}
\label{prop:Smu}

Assume that the sequences $\mu$ belong to the set $\{ \{\mu_j\}_{j \geq 0}, |\mu_j|\leq\frac{1}{n_j}\}$. Under condition (C), $\sup_{\{\mu, |\mu_j|\leq \frac{1}{n_j}\}}S_{\ell,m}(\mu)$ is uniformly bounded by $2^{-j_0\beta}+N^{-1/2} 2^{j_1/2}$ up to a multiplicative constant, that is, 
  $$\sup_{\mu\in\{ (\mu_j)_{j \geq 0},\, |\mu_j|\leq\frac{1}{n_j}\}} \{ S_{\ell,m}(\mu) \} = \BigO_\P (2^{-j_0\beta}+N^{-1/2} 2^{j_1/2}).$$
\end{prop}

\begin{proof}
$S_{\ell,m}(\mu)$ is decomposed in two terms $S^{(0)}_{\ell,m}(\mu)$ and $S^{(1)}_{\ell,m}(\mu)$,
\begin{align*}
S^{(0)}_{\ell,m}(\mu)&= \sum_{j=j_0}^{j_1} \mu_j \frac{1}{2^{j(d_\ell^0+d_m^0)}}\sum_k\left({W_{j,k}(\ell)W_{j,k}(m)}-\theta_{\ell,m}(j)\right),\\\
S^{(1)}_{\ell,m}(\mu)&=\sum_{j=j_0}^{j_1} n_j \mu_j\left[\frac{\theta_{\ell,m}(j)}{2^{j(d_\ell^0+d_m^0)}}-G_{\ell,m}^0 \right].
\end{align*} 
From Proposition~\ref{prop:cov_ondelettes},
\begin{align}
|S^{(0)}_{\ell,m}(\mu)|&\leq  \sum_{j=j_0}^{j_1} |\mu_j| \left|\sum_k{W_{j,k}(\ell)W_{j,k}(m)}-n_j\theta_{\ell,m}(j)\right|, \label{eqn:S0}\\
|S^{(1)}_{\ell,m}(\mu)|&\leq  C \sum_{j=j_0}^{j_1} 2^{-\beta j}n_j |\mu_j|. \label{eqn:S1}
\end{align} 
Under the assumption $|\mu_j|\leq \frac{1}{n_j}$, we have the inequality $|S^{(1)}_{\ell,m}(\mu)|\leq  C \sum_{j=j_0}^{j_1} 2^{-\beta j}$. The right-hand bound is equivalent to $2^{-j_0\beta}$ up to a constant.
Condition (C) gives
$$\E\left[\sup_{\{\mu, |\mu_j|\leq \frac{1}{n_j}\}} \left|S_{\ell,m}^{(0)}(\mu)\right|\right]\leq C'\sum_{j=j_0}^{j_1}  n_j^{-1/2},$$ with a positive constant $C'$.
As $n_j= N 2^{-j}(1+o(1))$ the right-hand side of the inequality is equivalent to $C'N^{-1/2} 2^{j_1/2}.$ 
\end{proof}

\begin{prop}
\label{prop:Smu2}
Let $0<j_0 \leq j_1 \leq j_N$.
Assume that the sequences $\mu$ belong to the set 
$$\mathcal{S}(q,\gamma,c)=\{ \{\mu_j\}_{j \geq 0}, |\mu_j|\leq\frac c n |j-j_0+1|^q 2^{(j-j_0)\gamma} \forall j=j_0,\ldots j_1\}$$ with $0\leq\gamma<1$. 
Under condition (C), $\sup_{\mu\in\mathcal{S}(q,\gamma,c)}S_{\ell,m}(\mu)$ is uniformly bounded by $2^{-j_0\beta}+H(N^{-1/2} 2^{j_0/2})$ up to a constant,
  $$\sup_{\mu\in\mathcal{S}(q,\gamma,c)}\{ S_{\ell,m}(\mu) \} =\BigO_\P (2^{-j_0\beta}+H_\gamma(N^{-1/2} 2^{j_0/2}))$$
  with $H_\gamma(u)=\begin{cases} u & \text{~if~~} 0\leq\gamma<1/2,\\
   \log(1+u^{-2})^{q+1}\,u & \text{~if~~} \gamma=1/2,\\
   \log(1+u^{-2})^q\, u^{2(1-\gamma)} & \text{~if~~} 1/2<\gamma<1.
  \end{cases}$
\end{prop}
In particular, for any $0\leq\gamma<1$, under the assumption $2^{-j_0\beta}+N^{-1/2} 2^{j_0/2}\to 0$, we have $\sup_{\mu\in\mathcal{S}(q,\gamma,c)}\{ S_{\ell,m}(\mu) \} =o_\P(1)$
\begin{proof}

Under the assumptions of the proposition, one deduce from inequality~\eqref{eqn:S1} that,  
$$\sup_{\mu\in\mathcal S(q,\gamma,c)}|S^{(1)}_{\ell,m}(\mu)|\leq  cC \frac 1 n \sum_{j=j_0}^{j_1} n_j 2^{(-\beta j +\gamma(j-j_0))}(j-j_0+1)^q\leq cC 2^{-\beta j_0} \sum_{i=0}^{j_1-j_0} 2^{-(1+\beta -\gamma)i}(i+1)^q.$$ 
The sum on the right-hand side of the inequality tends to 0 because $1+\beta-\gamma>0$.

Similarly, under the additional Condition (C), inequality~\eqref{eqn:S0} is rewritten as,
\begin{align*}
\E\left[\sup_{\mu\in\mathcal S(q,\gamma,c)} \left|S_{\ell,m}^{(0)}(\mu))\right|\right]&\leq cC'\frac{1}{n}\sum_{j=j_0}^{j_1} n_j^{1/2} 2^{\gamma (j-j_0)}(j-j_0+1)^q\\ 
&\leq cC' N^{-1/2}2^{j_0/2}\sum_{i=0}^{j_1-j_0} 2^{-(1/2-\gamma)i}(i+1)^q.
\end{align*}
We distinguish three cases depending on the values of $\gamma$.
\begin{itemize}
\item The result is straightforward when $0\leq\gamma<1/2$.
\item If $\gamma=1/2$, the right-hand side is bounded by $cC' N^{-1/2}2^{j_0/2}(j_1-j_0+1)^{q+1}$. The parameter $j_1$ satisfies $2^{j_1}\leq N$. Consequently $j_1-j_0\leq log_2(N)+log_2(2^{-j_0})=log_2(N2^{-j_0})$ and the result is proved.
\item When $1/2<\gamma<1$, the right-hand side admits the upper bound $cC'' N^{-1/2}2^{j_0/2}(j_1-j_0+1)^{q}2^{(\gamma-1/2)(j_1-j_0)}$ with $C''$ a positive constant. Since $2^{j_1}\leq N$, it is inferior to $cC'' (j_1-j_0+1)^{q}(N^{-1}2^{j_0})^{(1-\gamma)}$ which concludes the proof.
\end{itemize}
\vspace{-2\baselineskip}
\end{proof}

\section{Weak consistency}
\label{proof:consistance}

We first establish the convergence under the condition $2^{-j_0\beta}+N^{-1/2} 2^{j_1/2}\to 0$. This assumption is more restrictive than the condition $2^{-j_0\beta}+N^{-1/2} 2^{j_0/2}\to 0$ given in Theorem~\ref{prop:convergence}. Both conditions are equivalent when $j_1-j_0$ is finite but not in a general case.

We then prove Theorem~\ref{prop:convergence} in two steps: first we establish a lower bound for $\hat\bd$ and second we develop a proof similar to the first one that has been given in section~\ref{sec:conv1} but with a weaker assumption thanks to the previous bound.

\subsection{Consistency under non-optimal assumptions}
\label{sec:conv1}

In this section, we give a first result of consistency, with an assumption on $j_0$ and $j_1$ that can be weakened. This result is not necessary to obtain Theorem~\ref{prop:convergence} but the scheme of the proof is similar and it points out why an additional step is necessary.

\begin{prop}
\label{prop:conv1}
Assume that (W1)-(W5) and Condition (C) hold. If in addition $j_0$ and $j_1$ are chosen such that $2^{-j_0\beta}+N^{-1/2} 2^{j_1/2}\to 0$, then $$\hat \bd -\bd^0=o_\P(1),$$
$$\hat \bG(\hat \bd)-\bG(\bd^0)=o_\P(1).$$
\end{prop}

\begin{proof}

In order to evaluate the performances of the estimator of the long memory parameters, the first step consists in proving that the proposed estimator for $\bd$ is consistent. The equivalent properties for $\bOmega$ will be detailed in a second step.   
The proof is based on the following inequality, 
\begin{equation}
R(\bd)-R(\bd^0)\geq L(\bd-\bd^0)+\Delta(\bd,\bd^0),
\label{eqn:ineqR}
\end{equation}
where $L$ is a deterministic and convex function of $\bd$ and the remaining term $\Delta$ tends uniformly to zero in probability. 

We first establish inequality~\eqref{eqn:ineqR}. The difference between the criterion evaluated at $\bd$ and at the true long-memory parameters is equal to
$$
R(\bd)-R(\bd^0)= \log\det(\hat \bG(\bd))-\log\det(\hat \bG(\bd^0))+ 2\log(2)\mj\left(\sum_\ell d_\ell-d_\ell^0\right)$$
where $\mj=\frac{1}{n}\sum_{j=j_0}^{j_1} j\,n_j$ and $n=\sum_{j=j_0}^{j_1} n_j$.

The equality can be rewritten as
\begin{align*}
\lefteqn{R(\bd)-R(\bd^0)}\\
= & \log\det\left(\frac{1}{n} \sum_{j=j_0}^{j_1} \bLambda_{\mj}(\bd-\bd^0)\bLambda_j(\bd)^{-1} I(j) \bLambda_j(\bd)^{-1}\bLambda_{\mj}(\bd-\bd^0)\right)-\log\det(\hat \bG(\bd^0))\\
=&  \log\det\left(\frac{1}{n} \sum_{j=j_0}^{j_1} \bLambda_{j-\mj}(\bd-\bd^0)^{-1} \bLambda_j(\bd^0)^{-1} I(j) \bLambda_j(\bd^0)^{-1} \bLambda_{j-\mj}(\bd-\bd^0)^{-1}\right)\\
&  -\log\det(\hat \bG(\bd^0)).
\end{align*} Define $\lambda_j(\delta)=2^{-(j-\mj)\delta}$ for any $j\geq 0$ and $\delta\in\R$.

Let us first recall Oppenheim's inequality (see {\it e.g.} page 480 of \cite{MatrixBook}).
\begin{prop}[Oppenheim's inequality]
\label{prop:oppenheim}
Let $\bE$ and $\bB$ be two semi-definite positive matrices. Then $\det(\bE\circ \bB)\geq \det(\bE)\prod_\ell B_{\ell,\ell}$.
\end{prop}
Let $\bA$ be the following matrix, $$\bA=\frac{1}{n}\sum_{j=j_0}^{j_1} \bLambda_{j-\mj}(\bd-\bd^0)^{-1} \bLambda_j(\bd^0)^{-1} I(j) \bLambda_j(\bd^0)^{-1}\bLambda_{j-\mj}(\bd-\bd^0)^{-1}.$$ 
Oppenheim's inequality will be applied to matrices $\bB$ and $\bE(\bd-\bd^0)$ where the $(\ell,m)$-th element of $\bB$ is defined by $B_{\ell,m}=\frac{1}{n}\sum_{j=j_0}^{j_1} n_j \lambda_j(d_\ell-d_\ell^0)\lambda_j(d_m-d_m^0)$ and $\bE(\bd-\bd^0)=\bA\circ \tilde \bB$ where $\tilde B_{\ell,m}=B_{\ell,m}^{-1}$. The relation $\bA=\bE(\bd-\bd^0)\circ \bB$ holds. The $(\ell,m)$-th element of $\bE(\bd-\bd^0)$ is equal to $$E_{\ell,m}(\bd)=\sum_{j=j_0}^{j_1} \mu_{j,\ell,m}(\bd-\bd^0) I_{\ell,m}(j) 2^{-j(d_\ell^0+d_m^0)}$$ $$\text{~with~} \mu_{j,\ell,m}(\bdelta)=\frac{2^{-j(\delta_\ell+\delta_m)}2^{\mj(\delta_\ell+\delta_m)}}{\sum_{a=j_0}^{j_1} n_a 2^{-a(\delta_\ell+\delta_m)}2^{\mj(\delta_\ell+\delta_m)}}=\frac{2^{-j(\delta_\ell+\delta_m)}}{\sum_{a=j_0}^{j_1} n_a 2^{-a(\delta_\ell+\delta_m)}}.$$
\begin{itemize}
\item The matrix $\bE$ can be expressed as $\bE=\sum_{j,k}\tilde W_{j,k} \tilde W_{j,k}$. Consequently $\bE$ is positive semi-definite being the sum of positive semi-definite matrices. 
\item The matrix $\bB$ satisfies $\bB=\sum_{j=j_0}^{j_1} \bM_j \bM_j$ with $\bM_j=\left(\frac{n_j}{n}\right)^{1/2} \bLambda_{j-\mj}(\bd-\bd^0)^{-1}$. Thus $\bB$ is also positive semi-definite.
\end{itemize}

Oppenheim's inequality implies $\log\det(\bA)\geq \log\det(\bE(\bd-\bd^0))+\sum_\ell \log B_{\ell,\ell}$.

Define $L(\bd-\bd^0):=\sum_{\ell=1}^p \log B_{\ell,\ell}$. As we have $$\sum_{j=j_0}^{j_1} n_j \lambda_j(\delta)\lambda_j(\delta)=\sum_j n_j 2^{-2j\delta}2^{2\mj\delta}= 2^{2\mj\delta}\sum_{j=j_0}^{j_1} n_j 2^{-2 j\delta},$$
the function $L$ satisfies the following equality,
$$L(\bd-\bd^0)=\sum_{\ell=1}^p\left[\log(2^{2\mj(d_\ell-d_\ell^0)})+\log(\frac{1}{n}\sum_{j=j_0}^{j_1} n_j 2^{-2j(d_\ell-d_\ell^0)})\right].$$ It is easily seen that each term of the sum corresponds to the criterion defined in Proposition~6 of \cite{Moulines08Whittle}.

Inequality~\eqref{eqn:ineqR} follows with $\Delta(\bd,\bd^0)=\log\det(\bE(\bd-\bd^0))-\log\det(\hat \bG(\bd^0)).$ We will now control the two terms in the right-hand side inequality~\eqref{eqn:ineqR}. 

\begin{description}
\item[Control of $L$.] $L(\bd-\bd^0)$ is a multivariate extension of the criterion studied in Proposition~6 of \cite{Moulines08Whittle}. It is convex, positive and minimal at $\bd=\bd^0$.
\item[Control of $\Delta$.] We shall prove that both $\log\det \bE(\bd-\bd^0)$ and $\log\det(\hat \bG(\bd^0))$ tend uniformly to $\log\det(\bG^0)$ for $\bd\in\R^p$.

\begin{itemize}
\item The $(\ell,m)$-th element of the matrix $\bE(\bd-\bd^0)$ is equal to
 $$E_{\ell,m}(\bd-\bd^0)=\sum_{j=j_0}^{j_1} \mu_{j,\ell,m}(\bd-\bd^0)I_{\ell,m}(j) 2^{-j(d_\ell^0+d_m^0)}\text{~where~} \mu_{j,\ell,m}(\bdelta)=\frac{2^{-j(\delta_\ell+\delta_m)}}{\sum_a n_a 2^{-a(\delta_\ell+\delta_m)}}.$$
As $\sum_{j=j_0}^{j_1} n_j \mu_{j,\ell,m}(\bdelta)=1$, the quantity $E_{\ell,m}(\bd-\bd^0)$ is written as
$$E_{\ell,m}(\bd-\bd^0)=G_{\ell,m}^0+ \sum_{j,k} \mu_{j,\ell,m}(\bd-\bd^0) \left(\frac{W_{j,k}(\ell)W_{j,k}(m)}{2^{j(d_\ell^0+d_m^0)}}-G^0_{\ell,m}\right)$$
where $G_{\ell,m}^0 = \Omega_{\ell,m}K(d_\ell+d_m)\cos({\pi(d_\ell^0-d_m^0)/2})$.

Above expression is equal to $E_{\ell,m}(\bd-\bd^0)=G_{\ell,m}^0+S_{\ell,m}(\mu_{\ell,m}(\bd-\bd^0))$ with $S_{\ell,m}(\mu)$ defined previously in equation~\eqref{eqn:Smu}. Since $\sup_{\bd}|\mu_j(\bd-\bd^0)|\leq \frac{1}{n_j}$, Proposition~\ref{prop:Smu} states that $E_{\ell,m}(\bd-\bd^0) \to G_{\ell,m}^0$ uniformly in $\bd$ when $2^{-j_0\beta}+N^{-1/2}2^{j_1/2}\to 0$.

\item Finally we shall establish that $\log\det \hat \bG(\bd^0)$ tends to $\log\det(\bG^0)$. Recall $$\hat G_{\ell,m}(\bd^0)=G^0_{\ell,m}+S_{\ell,m}(\nu) \text{~where~} \nu_j=\frac{1}{n}.$$ The sequence $\nu$ belongs to the set $\mathcal S(0,0,1)$. Applying Proposition~\ref{prop:Smu2}, the convergence is proved when $2^{-j_0\beta}+N^{-1/2} 2^{j_0/2}\to 0.$
\end{itemize}
\end{description}
\vspace{-2\baselineskip}
\end{proof}

The consistency has been established in Proposition~\ref{prop:conv1} under the condition $2^{-j_0\beta}+N^{-1/2} 2^{j_1/2}\to 0$. The objective is to weaken this condition in order to prove Theorem~\ref{prop:convergence}. The scheme of the proof is a generalization of the proof of Proposition~9 of \cite{Moulines08Whittle} to multivariate cases.

The only step in the proof of Proposition~\ref{prop:conv1} that needs the assumption  $2^{-j_0\beta}+N^{-1/2} 2^{j_1/2}\to 0$ is the convergence study of $E_{\ell,m}(\bd-\bd^0)$ to $G_{\ell,m}^0$. The proof of Theorem~\ref{prop:convergence} consists in proving that $\hat \bd>\bd^0-1/2$ in probability in order to obtain a weaker convergence assumption for $E_{\ell,m}(\bd-\bd^0)$ applying Proposition~\ref{prop:Smu2}.

\subsection{Lower bound of the estimate}
 \label{sec:bound}
 
We proceed to show first that there exists $\bd^{min}$ such that for all $\ell=1,\ldots,p$, we have $d_\ell^0-1/2<d_\ell^{min}<d_\ell^0$ and $\P\left(\inf_{j_1\geq j_0+2} \hat d_\ell \leq d^{min}_\ell\right)$ tends to $0$ when $N$ goes to infinity. The proof is recursive. 
 
{\sc Step 1.}

We introduce $\tilde\alpha$ defined by 
$$\tilde\alpha_{j,\ell,m}=\begin{cases} \frac{1}{n} 2^{-(j-\mj)(d_\ell-d_\ell^0+d_m-d_m^0)}&\text{~if~~} j_0<j\leq \mj,\\  
\frac{1}{n} 2^{-(j-\mj)(d_\ell^{min}-d_\ell^0+d_m^{min}-d_m^0)} &\text{~if~~} \mj<j\leq j_1,\\
0 & \text{~otherwise} 
\end{cases}$$ where the vector $\bd^{min}$ is taken such that \begin{equation}\label{eqn:dmin}
\forall\ell=1,\ldots,p,~d_{\ell}^0-1/2<d_\ell^{min}<d_{\ell}^0 \text{~and~} \liminf_{n\to\infty}\inf_{\{\bd,~ d_{\ell}\leq d_\ell^{min}\}}\inf_{j_1=j_0,\dots,j_N}\sum_{j=j_0}^{j_1} n_j\tilde \alpha_{j,\ell,\ell}>1.
\end{equation}
Let $\ell$ be a given index in $\{1,\dots,p\}$. Following similar arguments as in \cite{Moulines08Whittle} when showing their formula (59), one can find $d_\ell^{min}$ satisfying condition~\eqref{eqn:dmin}.

The quantity $R(\bd)-R(\bd^0)$ is equal to
\begin{align*}
R(\bd)-R(\bd^0)&=\log\det \bA(\alpha(\bd-\bd^0))-\log\det \hat \bG(\bd^0)\\
\text{with~~} 
A_{\ell,m}(\alpha)&=\sum_{j=j_0}^{j_1} \alpha_{j,\ell,m}I_{\ell,m}(j)\text{~and~}\alpha_{j,\ell,m}(\delta)=\frac{1}{n}2^{-(j-\mj)(\delta_\ell+\delta_m)} .
\end{align*}
We first want to establish that for all $\bd$ in $\{\bd,~\forall \ell~ d_{\ell}\leq d_{\ell}^{min}\}$ we have $R(\bd)-R(\bd^0)\geq\log\det \bA(\tilde\alpha)-\log\det \hat \bG(\bd^0)$. To this aim we will use a generalization of Corollary 7.7.4 of \cite{MatrixBook}.
\begin{prop} \label{prop:det}
Let $\bA$ and $\bB$ be two positive semi-definite matrices of $\C^{p\times p}$. Suppose $\bA-\bB$ is positive semi-definite. Then $\det\bA\geq \det\bB$.
\end{prop}
\begin{proof}
We distinguish two cases:
\begin{itemize}
\vspace{-\baselineskip}
 \item Suppose $\det \bA=0$. Then there exists a unitary vector $x\in\R^p$ such that $x^T\bA x=0$. Using the fact that $\bA-\bB$ is positive semi-definite, we have $x^T(\bA-\bB)x=-x^T\bB x\geq 0$ which implies $x^T\bB x= 0$ since $\bB$ is positive semi-definite. Thus $\det \bB=0=\det \bA$.
\item  Suppose $\det \bA>0$. If $\det \bB=0$, inequality $\det \bB\leq \det \bA$ holds. If $\det \bB>0$, since $\bA-\bB$ is positive semi-definite we apply Corollary 7.7.4 of \cite{MatrixBook} which concludes the proof.\\ [-3\baselineskip]
\end{itemize}
\end{proof}
Moreover to prove the positive semi-definiteness of the matrices we will use the Schur product theorem (see {\it e.g.} page 458 of \cite{MatrixBook}).
\begin{prop}[Schur product theorem] \label{prop:schur}
Let $\bB_1$ and $\bB_2$ be two positive semi-definite matrices of $\C^{p\times p}$, then $\bB_1\circ \bB_2$ is also positive semi-definite.
\end{prop}

Let $j\geq 0$. The matrix $\bB_1(j)=(2^{j(d_\ell^0+d_m^0)}I_{\ell,m}(j))_{\ell,m}$ is positive semi-definite since it can be written as $\tilde W(j)^T\tilde W(j)$. The matrix $\bB_2=(\alpha_{j,\ell,m}(\bd-\bd^0)-\tilde\alpha_{j,\ell,m})_{\ell,m}$ has positive terms for all $\bd\in\{\bd, \forall \ell~ d_{\ell}\leq d_{\ell}^{min}\}$ and it is thus positive semi-definite.  Applying Proposition~\ref{prop:schur}, we obtain that $\bB(j)=\bB_1(j)\circ \bB_2(j)$ is positive semi-definite. $\bA(\alpha(\bd-\bd^0))-\bA(\tilde\alpha)$ is then positive semi-definite being the sum of positive semi-definite matrices. Similarly, it is easy to check that both $\bA(\tilde\alpha)$ and $\bA(\alpha(\bd-\bd^0))$ are positive semi-definite. 
Consequently Proposition~\ref{prop:det} gives $\log\det \bA(\alpha(\bd-\bd^0))\geq\log\det \bA(\tilde\alpha)$. This result holds for all $\bd$ satisfying $\forall \ell,\, d_{\ell}\leq d_{\ell}^{min}$.

We now study the behaviour of $\log\det \bA(\tilde\alpha)-\log\det \hat \bG(\bd^0)$. First, we have proved previously (see end of section~\ref{sec:conv1}) that $\log\det \hat \bG(\bd^0)$ tends uniformly in $\bd$ to $\log\det\bG^0$. Second, we decompose $\bA(\tilde\alpha)$ in $\bA(\tilde \alpha)=\tilde \bG +\bS(\tilde\alpha)$ with the elements of $\bS(\tilde\alpha)$ defined in equation~\eqref{eqn:Smu} and $\tilde \bG_{\ell,m}=\sum_j n_j\tilde \alpha_{j,\ell,m}G^0_{\ell,m}.$ We distinguish the study of the two terms.
\begin{itemize}
\item As $\mj\sim j_0$, for sufficiently large $N$ there exists a positive constant $c$ such that $$|\tilde \alpha_{j,\ell,m}|\leq \frac{c}{n}{2^{(j-j_0)(d_\ell^0-d_\ell^{min}+d_m^0-d_m^{min})}}.$$ Consequently, for all $(\ell,m)$, the sequence $\tilde \alpha_{\ell,m}$ belongs to $\mathcal S(0,\gamma,c)$ with $\gamma=2\sup_a (d_a^0-d_a^{min})$. As the vector $\bd^{min}$ satisfies that for any $\ell=1,\dots,p$, we have $d_\ell^0-1/2<d^{min}_\ell\leq d^0_\ell$, then $0\leq \gamma<1/2$. Applying Proposition~\ref{prop:Smu2} we deduce that $\bS(\tilde\alpha)$ tends to 0 in probability uniformly in $\bd$.

\item As $\bG^0$ is a covariance matrix, it is positive semi-definite. We can apply Oppenheim's inequality (Proposition~\ref{prop:oppenheim}), $\log\det\tilde \bG\geq \log\det(\bG^0)+\sum_{\ell}\log(\sum_{j=j_0}^{j_1} n_j\tilde \alpha_{j,\ell,\ell}).$ As we defined $\bd^{min}$ such that \eqref{eqn:dmin} holds, it follows that 
$\liminf_{N\to\infty}\inf_{\{\bd, \forall\ell~ d_\ell\leq d_\ell^{min}\}}\inf_{j_1=j_0,\dots,j_N} \log\det\bA(\tilde \alpha)-\log\det\bG^0>0.$
\end{itemize}
We thus get \begin{equation}\label{eqn:ineq_dmin}
\lim_{N\to\infty}\;\P\left(\inf_{\{\bd, \forall\ell~ d_\ell\leq d_\ell^{min}\}}\inf_{j_1=j_0,\dots,j_N} \log\det\bA(\tilde \alpha)-\log\det\hat \bG(\bd^0)>0\right)=1.
\end{equation}

Suppose $\hat d\in{\{\bd, \forall \ell~ d_{\ell}\leq d_{\ell}^{min}\}}$. By definition of $\hat \bd$, inequality $\inf_{\{\bd,~ \forall m~d_m\leq d_m^{min}\}} R(d)-R(d^0)\leq 0$ holds, which is in contradiction with result~\eqref{eqn:ineq_dmin}. Finally with a probability tending to $1$ there exists $\ell_1\in\{1,\dots,p\}$ such that $d_{\ell_1}\geq d_{\ell_1}^{min}$.

{\sc Step 2.} Suppose that exist $\ell_1,\ell_2,\ldots \ell_k$ with $k<p$ such that, with a probability tending to 1, $\hat d_{\ell_i}\geq d_{\ell_i}^{min}$ for all $i=1,\dots,k$. We introduce  $\tilde \alpha^{(k)}$ defined by $\tilde \alpha_{j,\ell,m}^{(k)}=\alpha_{j,\ell,m}(\bd-\bd^0)$ if $j_0\leq j\leq\mj$,  
$$\tilde \alpha_{j,\ell,m}^{(k)}=\begin{cases} \alpha_{j,\ell,m}(\bd-\bd^0) &\text{~if~~} \ell,m\in\{\ell_1,\ell_2,\ldots \ell_k\}\\
\alpha_{j,\ell,m}(\bd^{min}-\bd^0) &\text{~if~~} \ell\notin\{\ell_1,\ell_2,\ldots \ell_k\} \text{~and~} m\notin\{\ell_1,\ell_2,\ldots \ell_k\}\\
\frac{1}{n}{2^{-(j-\mj)(d_\ell-d_\ell^{0}+d_m^{min}-d_m^{0})}} &\text{~if~~} \ell\in\{\ell_1,\ell_2,\ldots \ell_k\} \text{~and~} m\notin\{\ell_1,\ell_2,\ldots \ell_k\}\\
\frac{1}{n}{2^{-(j-\mj)(d_\ell^{min}-d_\ell^{0}+d_m-d_m^{0})}} &\text{~if~~} \ell\notin\{\ell_1,\ell_2,\ldots \ell_k\} \text{~and~} m\in\{\ell_1,\ell_2,\ldots \ell_k\}\end{cases}$$
if $\mj<j\leq j_1$ and $\tilde \alpha_{j,\ell,m}^{(k)}=0$ else.
It is straightforward that for such $\tilde\alpha^{(k)}$ the three following points hold,
\begin{enumerate}
\item For all $\bd\in\{\bd,\,\forall\ell\notin\{\ell_1,\ldots,\ell_k\}~ d_\ell\leq d_\ell^{min}\}$, for all $j\geq 0$, for all $(\ell,m)\in\{1,\dots,p\}^2$, we have $\alpha_{j,\ell,m}(\bd-\bd^0)-\tilde\alpha^{(k)}_{j,\ell,m}\geq 0$;
\item For all $(\ell,m)\in\{1,\dots,p\}^2$, the sequence $(\tilde\alpha^{(k)}_{j,\ell,m})_{j\geq 0}$ belongs to $\mathcal S(0,\gamma,c)$ with $0\leq \gamma<1/2$;
\item $\liminf_{n\to\infty}\inf_{\{\bd, \forall\ell\notin\{\ell_1,\ldots,\ell_k\}~ d_\ell\leq d_\ell^{min}\}}\inf_{j_1=j_0,\dots,j_N}\sum_\ell\log(\sum_{j=j_0}^{j_1} n_j \tilde\alpha^{(k)}_{j,\ell_0,\ell_0})>0$.
\end{enumerate}
Analysis similar to {\sc Step 1} shows that there exists $\ell_{k+1}\notin\{\ell_1,\ell_2,\ldots \ell_k\}$ such that, with a probability tending to 1, $d_{\ell_{k+1}}\geq d_{\ell_{k+1}}^{min}$.

Step 1 and Step 2 imply that $\P(\hat\bd\in\{\bd,\,\forall\ell=1,\ldots,p,\, d_{\ell}\geq d_{\ell}^{min}\})\to 1$ when $N\to\infty$.

\subsection{Proof of Theorem~\ref{prop:convergence}}

Suppose first that for sufficiently large $N$ we have $j_1\geq j_0+2$. 

We have established that there exists $\bd^{min}$ such that for all $\ell=1,\ldots,p$, $d_\ell^0-1/2<d_\ell^{min}<d_\ell^0$ and $\P\left(\sup_{j_1\geq j_0+2} \hat d_\ell \leq d^{min}_\ell\right)$ tends to $0$ when $N$ goes to infinity.

Let $\bd$ be a $\R^p$-vector satisfying $d_\ell\geq d_\ell^{min}$ for all $\ell=1,\dots,p$. Recall that 
\begin{equation}\label{eqn:E}
E_{\ell,m}(\bd)=G^0_{\ell,m}+S_{\ell,m}(\mu_{\ell,m}(\bd-\bd^0))\text{~whith~} \mu_{j,\ell,m}(\bdelta)=\frac{2^{-j(\delta_\ell+\delta_m)}}{\sum_{j'\geq j_0} n_{j'} 2^{-j'(\delta_\ell+\delta_m)}}.
\end{equation} 
We bound the sequence $\mu$ as follows:\\
$$
|\mu_{j,\ell,m}(\bd-\bd^0)|\, \leq \,\frac{2^{(j-j_0)(d_\ell^0-d_\ell^{min}+d_m^0-d_m^{min})}}{\sum_{j'=j_0}^{j_1} n_{j'}2^{(j'-j_0)(d_\ell^0-d_\ell+d_m^0-d_m)}}\, \leq \,n_{j_0}^{-1} 2^{(j-j_0)(d_\ell^0-d^{min}_\ell+d_m^0-d^{min}_m)}.
$$
Since $n_{j_0}\sim N2^{-j_0}$ and $n\sim N2^{-j_0}(2-2^{j_1-j_0})$ there exist $N$ sufficiently large such that $n_{j_0}^{-1}\leq 2 n^{-1}$. We introduce $\gamma=\sup_{\ell=1,\dots,p}  d_\ell^0-d^{min}_\ell$. Then for all $(\ell,m)$, $\mu_{\ell,m}(\bd-\bd^0)$ belongs to $\mathcal S(0,2\gamma,2)$. We have $0\leq\gamma<1/2$ with a probability tending to 1. Applying Proposition~\ref{prop:Smu2} we deduce that $E_{\ell,m}(\bd-\bd^0) \to G_{\ell,m}^0$ uniformly in $\bd$ with a probability tending to 1. Following the proof of Proposition~\ref{prop:conv1} with this result, we obtain Theorem~\ref{prop:convergence}.

It remains to consider the case where we do not have $j_1\geq j_0+2$. Since we supposed $j_1-j_0$ is an increasing sequence of $N$, Proposition~\ref{prop:conv1} holds, which concludes the proof. 

\section{Rate of convergence}
 \label{proof:vitesse1}

In order to obtain the optimal rate of convergence, we first prove the convergence with a suboptimal rate. Based on this result, we are able to obtain feasible conditions under which the rates of Theorem~\ref{prop:vitesse} hold.

\subsection{Convergence with a suboptimal rate}

We shall now establish a first rate of convergence which is not optimal but which will be useful to derive conditions for the optimal rate. 

\begin{prop}
\label{prop:vitesse_sous_opt}
Assume that (W1)-(W5) and Condition (C) hold. If in addition $j_0$ is chosen such that $2^{-j_0\beta}+N^{-1/2} 2^{j_0/2}\to 0$ then $$\hat \bd -\bd^0=\BigO_\P(N^{-1/4}2^{j_0/4}+2^{-\beta j_0/2}).$$
\end{prop}

\begin{proof}
The proof is based on inequality~\eqref{eqn:ineqR}. The procedure is to find a lower bound for $R(\bd)-R(\bd^0)$ on the set $\{\bd,\, \max_{\ell=1,\dots,p}|d_\ell-d^0_\ell|< 1/4\}$. By Theorem~\ref{prop:convergence}, $\hat\bd-\bd^0$ goes to 0 in probability. Therefore, for sufficiently large $N$, we have $\max_{\ell=1,\dots,p}|\hat d_\ell-d^0_\ell|< 1/4$.

First, a second order Taylor expansion of $L(\bd-\bd^0)$ at the neighbourhood of $0$ for $\bd^0-1/4<\bd\leq\bd^0$ is $L(\bd-\bd^0)=\left.\frac{d^2L(u)}{du^2}\right|_{\bar u} (\bd-\bd^0)^2+o(\max_{\ell=1,\dots,p}|d_\ell-d^0_\ell|^2)$
with $|\bar u|\leq 1/4$. Proposition~6 of \cite{Moulines08Whittle} states that $\liminf_{N\to\infty}\inf_{\bar u\in[-1/2,0]}\inf_{j_1=j_0+1,\dots,j_N} \left.\frac{d^2L(u)}{du^2}\right|_{\bar u}>0$. Thus for $\bd^0-1/4<\bd\leq\bd^0+1/4$, there exists $c>0$ such that $L(\bd-\bd^0)>c\sum_{\ell=1}^p(d_\ell-d_\ell^0)^2+o(\max_{\ell=1,\dots,p}|d_\ell-d^0_\ell|^2)$.

We now want to establish an upper bound for $\Delta(\bd,\bd^0)$.  
This quantity satisfies $$\Delta(\bd,\bd^0)=\log\det(I+\bG^{0-1}\bS(\mu(\bd-\bd^0))))$$  where $\bS(\mu(\bd-\bd^0))$ is the matrix with $(\ell,m)$-th element $S_{\ell,m}(\mu(\bd-\bd^0))$ defined in~\eqref{eqn:Smu} and $\mu_{j,\ell,m}(\cdot)$ defined in~\eqref{eqn:E}. Thus, $\Delta(\bd,\bd^0)=\log(\prod_{i=1}^p\lambda_i)$ where $(\lambda_i)_{i=1,\dots,p}$ denote the eigenvalues of the semi-definite positive matrix $\bG^{0-1}\bS(\mu(\bd-\bd^0))$. Hence $$0\leq\Delta(\bd,\bd^0)\leq trace(\bG^{0-1}\bS(\mu(\bd-\bd^0))).$$
Since $\max_{\ell=1,\dots,p} |d_\ell-d^{0}_\ell|\to0$, for $N$ sufficiently large the quantity $\gamma=\max_{\ell=1,\dots,p} |d_\ell-d^{0}_\ell|$ satisfies $0\leq\gamma<1/4$. As established previously, for all $(\ell,m)$, $\mu_{\ell,m}(\bd-\bd^0)$ belongs to $\mathcal S(0,2\gamma,2)$. Then for all $(\ell,m)$, $S_{\ell,m}(\mu_{\ell,m}(\bd-\bd^0))=\BigO_\P(2^{-j_0\beta}+N^{-1/2} 2^{j_0/2})$ uniformly in $\bd\in[\bd^0-1/4,\bd^0]$ applying Proposition~\ref{prop:Smu2}. Hence $\Delta(\bd,\bd^0)=\BigO_\P(2^{-j_0\beta}+N^{-1/2} 2^{j_0/2}).$  

Inequality~\eqref{eqn:ineqR} thus gives $$R(\bd)-R(\bd^0)\geq c\sum_{\ell=1}^p(d_\ell-d_\ell^0)^2+o(\max_\ell|d_\ell-d_\ell^0|^2) + \BigO_\P(2^{-j_0\beta}+N^{-1/2} 2^{j_0/2})$$ 
with $c>0$. It follows that for all $\ell=1,\dots,p$, $(\hat d_\ell-d_\ell^0)^2=\BigO_\P(2^{-j_0\beta}+N^{-1/2} 2^{j_0/2})$.
\end{proof}

\subsection{Proof of Theorem~\ref{prop:vitesse}}

The criterion $R$ is equal to $R(\bd)=\log\det \left(\bLambda_{\mj}(\bd)\hat \bG(\bd) \bLambda_{\mj}(\bd)\right)-1$. It is straightforward that $\hat \bd = \argmin_{\bd} R(\bd)$ satisfies \begin{eqnarray}
\label{eqn:R2}\hat \bd = \argmin_{\bd} \barR(\bd)  &\text{~with~}& \barR(\bd)=\log\det\barG(\bd)\\ \nonumber &\text{~and~}& \barG(\bd)=\bLambda_{\mj}(\bd-\bd^0)\hat \bG(\bd) \bLambda_{\mj}(\bd-\bd^0)
\end{eqnarray}
The Taylor expansion of $\barR$ at $\hat \bd$ at the neighbourhood of $\bd^0$ gives \begin{equation}\label{eqn:taylor}
R(\hat\bd)-R(\bd^0)= \left.\frac{\partial \barR(\bd)}{\partial \bd}\right|_{\bd^0}(\bd-\bd^0) + (\hat \bd -\bd^0)^T\left. \frac{\partial^2 \barR(\bd)}{\partial \bd \partial \bd^T}\right|_{\bar{\bd}}(\hat \bd -\bd^0)
\end{equation}
 where $\bar{\bd}$ is such that $\|\bar{\bd}-\bd^0\|\leq\|\hat \bd-\bd^0\|$.

The derivatives of the criterion $\barR(\bd)$ are equal to
\begin{align}
\label{eqn:derivR}\frac{\partial \barR(\bd)}{\partial d_a}&= trace\left(\barG(\bd)^{-1} \frac{\partial \barG(\bd)}{\partial d_a} \right)\\
\label{eqn:deriv2R}\frac{\partial^2 \barR(\bd)}{\partial d_a\partial d_b}&= - trace\left(\barG(\bd)^{-1} \frac{\partial \barG(\bd)}{\partial d_b}\barG(\bd)^{-1} \frac{\partial \barG(\bd)}{\partial d_a} \right)+trace\left(\barG(\bd)^{-1} \frac{\partial^2 \barG(\bd)}{\partial d_a \partial d_b} \right)
\end{align}
when $\barG(\bd)^{-1}$ exists.

\subsubsection[Convergence of G and its derivatives]{Study of $\barG$ and its derivatives}

To study the asymptotic behaviour of the derivatives of the criterion it is necessary first to study the asymptotic behaviour of $\barG$ and of its derivatives.

For any $a=1,\ldots,p$, let $\bi_a$ be a $p\times p$ matrix whose a-{\it th} diagonal element is one and all other elements are
zero. Let $a$ and $b$ be two indexes in $1,\ldots,p$. The first derivative of $\barG(\bd)$ with respect to $d_a$, $\frac{\partial \barG(\bd)}{\partial d_a}$, is equal to
$$ - \log(2) \frac{1}{n} \sum_{j=j_0}^{j_1} (j-\mj)\bLambda_{\mj}(\bd-\bd^0){\bLambda_{j}(\bd)}^{-1}(\bi_a \bI(j)+\bI(j)\bi_a){\bLambda_{j}(\bd)}^{-1}\bLambda_{\mj}(\bd-\bd^0).$$
And the second derivative,  with respect to $d_a$ and $d_b$,
\begin{multline}
\label{eqn:deriv2G}
\frac{\partial^2 \barG(\bd)}{\partial d_a\partial d_b}=\log(2)^2 \bLambda_{\mj}(\bd-\bd^0)\frac{1}{n} \sum_{j=j_0}^{j_1} (j-\mj)^2\bLambda_{j}(\bd)^{-1}\\  \left(\bi_b \bi_a \bI(j) + \bI(j)\bi_a \bi_b + \bi_b \bI(j)\bi_a + \bi_a \bI(j)\bi_b\right)\bLambda_{j}(\bd)^{-1}\Lambda_{\mj}(\bd-\bd^0)
\end{multline}

\subsubsection*{Convergence of $\barG(\bd)$}

Let $\ell,m$ be given indexes in $\{1,\ldots, p\}$.
Any $(\ell,m)$-th element of the matrix $\barG(\bd)$ satisfies
$$\bar G_{\ell,m}(\bd)=G^0_{\ell,m}\sum_{j=j_0}^{j_1} n_j \omega^{(0)}_{j,\ell,m}+S_{\ell,m}(\omega^{(0)}_{\ell,m}(\bd-\bd^0))$$
where $\omega^{(0)}_{j,\ell,m}(\bd-\bd^0)=\frac{1}{n} 2^{-(j-\mj)(d_\ell-d_\ell^0+d_m-d_m^0)}.$

Recall that $\mj=(j_0+\eta_{j_1-j_0})(1+o(1))$ with $0\leq\eta_{j_1-j_0}\leq 1$. Let $\bd$ be a $\R^p$-vector such that for all $\ell=1,\dots,p$, $\hat d_\ell\leq d_\ell\leq d^0_\ell$. As $\sup_\ell|\hat d_\ell-d_\ell^0|=o_\P(1)$, for any $\gamma\in(0,1/2)$, there exists $N_\gamma$ such that for any $N \geq N_\gamma$, $2^{-(j-\mj)(d_\ell-d_\ell^0+d_m-d_m^0)}\leq 2^\gamma 2^{(j-j_0)\gamma}$. For $N\geq N_\gamma$, the sequence $\omega^{(0)}_{\ell,m}(\hat \bd-\bd^0)$ belongs to $\mathcal{S}(0,\gamma,2^\gamma)$. Proposition~\ref{prop:Smu2} shows that $S_{\ell,m}(\omega^{(0)}_{\ell,m}(\bd-\bd^0))$ tends to zero when $2^{-j_0\beta}+N^{-1/2} 2^{j_0/2}\to 0$ uniformly in $\bd$.

Finally, we shall prove that $\sum_{j=j_0}^{j_1} n_j \omega^{(0)}_{j,\ell,m}(\bd-\bd^0)\to 1$. Since $|2^{a}-1|\leq 2^{|a|}-1\leq \log(2)|a|2^{|a|}$ for all $a\in\R$,
\begin{align*}
|\sum_{j=j_0}^{j_1} n_j \omega^{(0)}_{j,\ell,m}(\bd-\bd^0)- 1|& \leq \frac 1 n \sum_j n_j  \log(2)|j-\mj|\max_{\ell=1,\dots,p}|d_\ell-d^0_\ell| 2^{|j-\mj|\max_{\ell=1,\dots,p}|d_\ell-d^0_\ell|} \\
& \leq  2\log(N) \max_{\ell=1,\dots,p}|d_\ell-d^0_\ell| 2^{2\log_2(N)\max_{\ell=1,\dots,p}|d_\ell-d^0_\ell|}.
\end{align*}
The last inequality is a consequence of $j_0\leq j_1\leq j_N=\log_2(N)$. 
Under assumption $\log(N)\max_{\ell=1,\dots,p}|\hat d_\ell-d^0_\ell|\to 0$, the right-hand side of the inequality goes to 0 when $N$ goes to infinity.

Consequently,
\begin{equation}\label{eqn:convGbar}
\sup_{\{\bd,\,\|\bd-\bd^0\|\leq\|\hat\bd-\bd^0\|\}} \bar G_{a,b}(\bd)=G^0_{a,b}(1+o_\P(\log(N)\max_\ell|d_\ell-d_\ell^0|))+o_\P(2^{-j_0\beta}+N^{-1/2}2^{j_0/2}).
\end{equation}
Due to Proposition~\ref{prop:vitesse_sous_opt}, it is sufficient that $\log(N)^2(2^{-j_0\beta}+N^{-1/2}2^{j_0/2})\to 0$ to obtain $\sup_{\{\bd,\,\|\bd-\bd^0\|\leq\|\hat\bd-\bd^0\|\}} \bar G_{a,b}(\bd)=G^0_{a,b}+o_\P(1).$
Thus for sufficiently large $N$, $\barG(\bd)$ is invertible and $\barG(\bd)^{-1}$ converges in probability to $\bG^{0-1}$ on the set ${\{\bd,\,\|\bd-\bd^0\|\leq\|\hat\bd-\bd^0\|\}}$.

\subsubsection*{Convergence of $\left. \frac{\partial \barG(d)}{\partial d_a}\right|_{\bd}$}

This section concerns the convergence in probability of $\left(\left.\frac{\partial \barG(\bd)}{\partial d_a}\right|_{\bd}\right)_{a,b}$ which is equal to
\begin{align*}
\left(\left.\frac{\partial \barG(\bd)}{\partial d_a}\right|_{\bd}\right)_{a,b}&=\log(2)\frac 1 n \sum_{j=j_0}^{j_1} (j-\mj) 2^{-j(d_a^0+d_b^0)}I_{a,b}(j)\\
&= \log(2)\left[G^0_{a,b}\sum_{j=j_0}^{j_1} n_j \omega^{(1)}(\bd-\bd^0)+ S_{a,b}(\omega^{(1)}(\bd-\bd^0))\right],
\end{align*}
where $\omega^{(1)}_{j}(\bdelta)= \frac 1 n (j-\mj)2^{-(j-\mj)(\delta_\ell+\delta_m)}$.

We first study the behaviour of $\sum_{j=j_0}^{j_1} n_j \omega^{(1)}(\bd-\bd^0)$.
\begin{align*}
|\sum_{j=j_0}^{j_1} n_j \omega^{(1)}(\bd-\bd^0)|& \leq \frac 1 n \sum_j n_j  \log(2)|j-\mj|\max_{\ell=1,\dots,p}|d_\ell-d^0_\ell| 2^{|j-\mj|\max_{\ell=1,\dots,p}|d_\ell-d^0_\ell|} \\
&\leq \log(N)\max_{\ell=1,\dots,p}|d_\ell-d^0_\ell| 2^{\log_2(N)\max_{\ell=1,\dots,p}|d_\ell-d^0_\ell|}
\end{align*}
It is thus sufficient that $\log(N)\max_{\ell=1,\dots,p}|\hat d_\ell-d^0_\ell|\to 0$ to have $\sum_{j=j_0}^{j_1} n_j \omega^{(1)}(\bd-\bd^0)=o_\P(1)$.

Let $\bd$ be in a neighbourhood of $\bd^0$ such that $\sup_\ell |d_\ell-d^0_\ell|<\gamma$ with $0\leq\gamma<1/2$. As $\mj\sim j_0+\eta_{j_1-j_0}$ with $0\leq \eta_{j_1-j_0}\leq 1$, there exists $N_0$ such that for any $N \geq N_0$ the sequence $\omega^{(1)}(\bd-\bd^0)$ belongs to $\mathcal S(1,\gamma,2^\gamma)$. Thanks to Proposition~\ref{prop:Smu2} it comes that $S_{a,b}(\omega^{(1)}(\bd-\bd^0))=\BigO_\P\left(2^{-j_0 \beta}+N^{-1/2}2^{j_0/2}\right)$ uniformly on the neighbourhood. Consequently, $\left(\left.\frac{\partial \barG(\bd)}{\partial d_a}\right|_{\bd}\right)_{a,b}=o_\P(1)$. Finally considering similarly the other terms, when $\log(N)\max_{\ell=1,\dots,p}|\hat d_\ell-d^0_\ell|\to 0$, the equivalence $\left.\frac{\partial \barG(\bd)}{\partial d_a}\right|_{\bd} = o_\P(1)$ is fulfilled
under assumptions of Theorem~\ref{prop:vitesse}.

In addition, when $\bd=\bd^0$ we have $\sum_{j=j_0}^{j_1} n_j \omega^{(1)}_j(\mathbf{0})=0$. Hence,
$\left(\left.\frac{\partial \barG(\bd)}{\partial d_a}\right|_{\bd^0}\right)_{a,b}= C^{(1)}_{a,b}
$
where $C_{a,b}^{(1)}=\BigO_\P\left(2^{-j_0 \beta}+N^{-1/2}2^{j_0/2}\right)$.
Or more generally,
\begin{equation}\label{eqn:dG}
\left.\frac{\partial \barG(\bd)}{\partial d_a}\right|_{\bd^0} = \bi_a \bC^{(1)}+ \bC^{(1)}\bi_a,
\end{equation}
where each term of the matrix $\bC^{(1)}$ is $\BigO_\P\left(2^{-j_0 \beta}+N^{-1/2}2^{j_0/2}\right)$.

\subsubsection*{Convergence of $\left.\frac{\partial^2 \barG_{a,b}(\bd)}{\partial d_a\partial d_b}\right|_{\bd}$}

Let $\bd$ be a $\R^p$-vector such that $|d_\ell-d_\ell^0|\leq|\hat d_\ell-d_\ell^0|$ for all $\ell=1,\dots,p$. The proof is derived for $\left(\left.\frac{\partial^2 \barG_{a,b}(\bd)}{\partial d_a\partial d_b}\right|_{\bd}\right)_{a,b}$, for $a\neq b$. The argumentation is similar for the diagonal terms. Introducing a sequence $S_{a,b}(\cdot)$, the expression~\eqref{eqn:deriv2G} is rewritten as
\begin{align*}
\left(\left.\frac{\partial^2 \barG_{a,b}(\bd)}{\partial d_a\partial d_b}\right|_{\bd}\right)_{a,b}&=\log(2)^2 \left[2^{\mj( d_a+  d_b-d_a^0-d_b^0)}\frac{1}{n}\sum_{j=j_0}^{j_1} (j-\mj)^2 \frac{I_{a,b}(j)}{2^{j( d_a+ d_b)}}\right]\\
&= \log(2)^2\left[ G_{a,b}^0\sum_{j=j_0}^{j_1} n_j \omega^{(2)}_{j,a,b}(\bd-\bd^0) + S_{a,b}(\omega^{(2)}_{a,b}(\bd-\bd^0))\right]
\end{align*}
where $\omega^{(2)}_{j,a,b}(\bdelta)=\frac{1}{n}(j-\mj)^2 2^{-(j-\mj)(\delta_a+\delta_b)}.$

First, we want to prove that 
\begin{equation}\label{eqn:sum_w}
\sum_{j=j_0}^{j_1} n_j \omega_{j,a,b}^{(2)}(\bd-\bd^0) = \kappa_{j_1-j_0}(1+o_\P(1)).
\end{equation}
Recall $\kappa_{j_1-j_0}=\frac{1}{n} \sum_{j=j_0}^{j_1} (j-\mj)^2 n_j$.
Since $|2^{a}-1|\leq 2^{|a|}-1\leq \log(2)|a|2^{|a|}$ for all $a\in\R$, we have the inequality
$$
|\sum_{j=j_0}^{j_1} n_j \omega_{j,a,b}^{(2)}(\bar\bd-\bd^0) - \kappa_{j_1-j_0}| \leq  2\log(2)  \kappa_{j_1-j_0} \log_2(N) |d-d^0| 2^{2\log_2(N)|d-d^0|}.
$$ 
Equation~\eqref{eqn:sum_w} holds when $\log(N) \max_{\ell=1,\dots,p}|d_\ell-d^0_\ell|\to 0$. Due to Proposition~\ref{prop:vitesse_sous_opt} it is sufficient that $\log(N)^2(2^{-j_0\beta}+N^{-1/2}2^{j_0/2})\to 0$.

Let $\bd$ be in a neighbourhood of $\bd^0$ such that $\sup_\ell |d_\ell-d^0_\ell|<\gamma$ with $0\leq\gamma<1/2$. As $\mj\sim j_0+\eta_{j_1-j_0}$ with $0\leq \eta_{j_1-j_0}\leq 1$, there exists $N_0$ such that for any $N \geq N_0$ the sequence $\omega^{(2)}_{a,b}(\bd-\bd^0)$ belongs to the set $\mathcal S(2,\gamma,2^\gamma)$. Using Proposition~\ref{prop:Smu2}, $S_{a,b}(\omega^{(2)}_{a,b}(\bd-\bd^0))\leq C ( 2^{-j_0 \beta}+N^{-1/2}2^{j_0/2})$ uniformly on the neighbourhood of $\bd$ for $N\geq N_0$. As a consequence, \begin{equation*}\left(\left.\frac{\partial^2 \barG_{a,b}(\bd)}{\partial d_a\partial d_b}\right|_{\bar\bd}\right)_{a,b}= \log(2)^2\kappa_{j_1-j_0} G_{a,b}^0\left(1+o_\P(1)\right) + C^{(2)}_{a,b},
\end{equation*}
with $C^{(2)}_{a,b}=\BigO_\P( 2^{-j_0 \beta}+N^{-1/2}2^{j_0/2}).$

Finally, when $N$ goes to infinity, 
\begin{equation}\label{eqn:d2G}
\sup_{\{\bd,\,\|\bd-\bd^0\|\leq\|\hat\bd-\bd^0\|\}}\left.\frac{\partial^2 \barG_{a,b}(\bd)}{\partial d_a\partial d_b}\right|_{\bar\bd} = \log(2)^2 \kappa_{j_1-j_0} \left(\bi_b \bi_a\bG^0+\bi_b\bG^0\bi_a +\bi_a\bG^0\bi_b+\bG^0\bi_a \bi_b\right)+o_\P(1)
\end{equation}
when $\log(N)\max_{\ell=1,\dots,p}|\hat d_\ell-d^0_\ell|=o_\P(1)$ and $2^{-j_0 \beta}+N^{-1/2}2^{j_0/2}\to 0$.

\subsubsection{Second derivative of the criterion}

Let us detail the expression of the second derivative of $R$ with respect to $\bd$ at $\bar\bd$, $$\left.\frac{\partial^2 \barR(\bd)}{\partial d_a\partial d_b}\right|_{\bar\bd}= - trace\left(\barG(\bar\bd)^{-1} \left.\frac{\partial \barG(\bd)}{\partial d_b}\right|_{\bar\bd}\barG(\bar\bd)^{-1} \left.\frac{\partial \barG(\bd)}{\partial d_a}\right|_{\bar\bd} \right)+trace\left(\barG(\bar\bd)^{-1} \left.\frac{\partial^2 \barG(\bd)}{\partial d_a \partial d_b}\right|_{\bar\bd} \right).$$
Using \eqref{eqn:convGbar} and the previous study that established that $\left.\frac{\partial \barG(\bd)}{\partial d_a}\right|_{\bar\bd} = o_\P(1)$, the first term tends to~0 under assumptions of Theorem~\ref{prop:vitesse}. Combining~\eqref{eqn:convGbar} and~\eqref{eqn:d2G} we can assert that $\left.\frac{\partial^2 \barR(\bd)}{\partial d_a\partial d_b}\right|_{\bar\bd}$ tends in probability to $$\log(2)^2\kappa_{j_1-j_0} trace\left(\bG^{0-1} (\bi_b \bi_a\bG^0+\bi_b\bG^0\bi_a +\bi_a\bG^0\bi_b+\bG^0\bi_a \bi_b) \right).$$
Let $G^{0(-1)}_{\ell,m}$ denotes the $(\ell,m)$-th element of $\bG^{0-1}$.
When $a\neq b$, $$trace\left(\bG^{0-1} (\bi_b \bi_a\bG^0+\bi_b\bG^0\bi_a +\bi_a\bG^0\bi_b+\bG^0\bi_a \bi_b) \right)=G^{0(-1)}_{a,b} G^0_{a,b}+G^{0(-1)}_{b,a} G^0_{b,a}=2 G^{0(-1)}_{a,b} G^0_{a,b}.$$
When $a=b$, $$trace\left(\bG^{0-1} (\bi_b \bi_a\bG^0+\bi_b\bG^0\bi_a +\bi_a\bG^0\bi_b+\bG^0\bi_a \bi_b) \right)= 2(1+ G^{0(-1)}_{a,a} G^0_{a,a}).$$
Finally, 
\begin{equation}\label{eqn:d2R}
\left.\frac{\partial^2 \barR(\bd)}{\partial \bd\partial \bd^T}\right|_{\bar\bd}=  \kappa_{j_1-j_0} 2\log(2)^2\left( \bG^{0-1}\circ \bG^0 + \bI_p\right)+o_\P(1).
\end{equation}
The matrix $\bG^{0-1}\circ \bG^0 $ is positive definite using Schur product theorem (Proposition~\ref{prop:schur}) and hence the matrix $\bG^{0-1}\circ \bG^0 + \bI_p$ is invertible.

\subsubsection{End of the proof}
 \label{proof:vitesse2}

The Taylor expansion~\eqref{eqn:taylor} together with \eqref{eqn:d2R} imply
\begin{align*}
2\log(2)^{2}\kappa_{j_1-j_0}(\hat \bd - \bd^0) &=  -\left(\frac{1}{2\log(2)^2\kappa_{j_1-j_0}}\left.\frac{\partial^2 \barR(\bd)}{\partial \bd\partial \bd^T}\right|_{\bar \bd}\right)^{-1} \left.\frac{\partial \barR(\bd)}{\partial \bd}\right|_{\bd^0}\\
&= (\bG^{0-1}\circ \bG^0+\bI_p)^{-1}\left.\frac{\partial \barR(\bd)}{\partial \bd}\right|_{\bd^0} (1+o_\P(1))
\end{align*}
We now study the convergence of $\left.\frac{\partial \barR(\bd)}{\partial \bd}\right|_{\bd^0}$. Using equations~\eqref{eqn:derivR} and~\eqref{eqn:dG}, we have the equation
\begin{equation*}
\left.\frac{\partial \barR(\bd)}{\partial d_a}\right|_{\bd^0} = 2\sum_{b=1}^p \barG_{a,b}^{(-1)}(\bd^0)C^{(1)}_{a,b}.
\end{equation*}
So the asymptotic behaviour of the first derivative of the criterion is  $$\left.\frac{\partial \barR(\bd)}{\partial d_a}\right|_{\bd^0}=\BigO_\P( 2^{-j_0 \beta}+N^{-1/2}2^{j_0/2}).$$
Plugging this result into the expression above, it comes $$ \kappa_{j_1-j_0}\left(2^{-j_0 \beta}+N^{-1/2}2^{j_0/2}\right)^{-1} (\hat \bd -\bd^0)=\BigO_\P(1).$$
Since $\kappa_{\ell}>0$ for $\ell\geq 1$ and $\kappa_\ell\to 2$ when $\ell\to\infty$, the sequence $\kappa_{j_1-j_0}$ is bounded below by a positive constant. The rate of convergence for $\hat\bd-\bd^0$ in Theorem~\ref{prop:vitesse} follows.

\subsection[Convergence of the long-run covariance matrix]{Convergence of $\hat\bG(\hat\bd)$ and of $\hat\bOmega$}

Recall $\hat\bG_{\ell,m}(\hat\bd)=2^{\mj(\hat d_\ell-d_\ell^0+\hat d_m-d_m^0)}\bar G_{\ell,m}(\hat \bd).$
Equation~\eqref{eqn:convGbar} with the rate obtained for the convergence of $\hat\bd-\bd^0$ state that $\bar G_{\ell,m}(\hat \bd)=G^0_{\ell,m}(1+\BigO_\P(\log(N)(2^{-j_0 \beta}+N^{-1/2}2^{j_0/2}))+\BigO_\P\left(2^{-j_0 \beta}+N^{-1/2}2^{j_0/2}\right)$ under assumptions of Theorem~\ref{prop:vitesse}. The rate of convergence of $\hat G_{\ell,m}(\hat\bd)$ in Theorem~\ref{prop:vitesse} is then derived from the fact that $2^{\mj u}-1= j_0 u\log(2)(1+o(1))$ when $u\to 0$.

The convergence of $\hat\bOmega$ is straightforward, thanks to the fact that $K(\cdot)$ is a continuous function of $\bd$.
To obtain the rate of convergence, we observe first that $\cos(u)=1+o(u^2)$ when $u$ goes to $0$. Second, 
\begin{align*}
K(\hat d_\ell+\hat d_m)-K(d_\ell^0+d_m^0)&=\int_{-\infty}^\infty{(|\lambda|^{-\hat d_\ell-\hat d_m}-|\lambda|^{-d_\ell^0-d_m^0})|\hat\psi(\lambda)|^2d\lambda}\\
& \leq  |\hat d_\ell+\hat d_m-d_\ell^0-d_m^0|\int_{-\infty}^\infty|\log|\lambda|||\lambda|^{-d_\ell^0-d_m^0}|\hat\psi(\lambda)|^2d\lambda.
\end{align*} 
Using assumption (W2) and (W5), $$|K(\hat d_\ell+\hat d_m)-K(d_\ell^0+d_m^0)|\leq |\hat d_\ell+\hat d_m-d_\ell^0-d_m^0|\, C \int_{-\infty}^\infty|\log|\lambda||\,|\lambda|^{-(1+\beta)} d\lambda$$ with C positive constant. The integral on the right-hand side is finite and thus $K(\hat d_\ell+\hat d_m)-K(d_\ell^0+d_m^0)=\BigO_\P(\max_{i=1,\dots,p} |\hat d_i-d_i^0|)$. When $\max_{i=1,\dots,p} |\hat d_i-d_i^0|=\BigO_\P\left(2^{-j_0 \beta}+N^{-1/2}2^{j_0/2}\right)$, we have $1/(K(\hat d_\ell+\hat d_m)cos(\frac{\pi}{2}(\hat d_\ell-\hat d_m)))=(1+\BigO_\P\left(2^{-j_0 \beta}+N^{-1/2}2^{j_0/2}\right))$ which concludes the proof.

\subsection{Additional tools}

These results correspond to Lemma~13 of \cite{Moulines08Whittle}

Define the sequences $\eta_L$ and $\kappa_L$ for any $L\geq 0$ by \begin{align}
\label{eqn:eta}
\eta_L &:= \sum_{i=0}^L i \frac{2^{-i}}{2-2^{-L}}\\
\label{eqn:kappa}
\kappa_L &:= \sum_{i=0}^L (i-\eta_L)^2 \frac{2^{-i}}{2-2^{-L}}
\end{align}
It is straightforward that
\begin{align}
n &\sim N 2^{-j_0}(2-2^{-(j_1-j_0)})\\
<\mathcal J> &\sim j_0 + \eta_{j_1-j_0}\\
\frac{1}{n}\sum_{j=j_0}^{j_1}(j-<\mathcal J>)^2 n_j &\sim \kappa_{j_1-j_0}
\end{align}
For every $L\geq 1$ the quantities $\eta_L$ and $\kappa_L$ are strictly positive. When $L$ goes to infinity, the sequences $\eta_L$ and $\kappa_L$ respectively converge to $1$ and $2$.

And for all $u\geq 0$,  \begin{equation}
\frac{1}{\kappa_L}\sum_{i=0}^{L-u}\frac{2^{-i}}{2-2^{-L}}(i-\eta_L)(i+u-\eta_L) \to 1 \text{~when~} L\to\infty
\end{equation}

\end{appendices}

\setlength{\parskip}{0pt}
\bibliographystyle{spbasic}
\bibliography{biblio}

\end{document}